\documentclass{amcjou}

\usepackage{amssymb,relsize}

\usepackage{graphicx}

\usepackage{hyperref}
\usepackage{cleveref}

\newcommand{\pref}[1]{(\ref{#1})}
\newcommand{\csee}[1]{(see \cref{#1})}

\newcommand{\fullcref}[2]{\cref{#1}\pref{#1-#2}}
\newcommand{\fullcsee}[2]{(see \fullcref{#1}{#2})}

\newcommand{\integer}{\mathbb{Z}}
\renewcommand{\natural}{\mathbb{N}}
\newcommand{\iso}{\cong}
\newcommand{\normal}{\triangleleft}

\newcommand{\quot}[1]{\overline{#1}}
\newcommand{\voltage}{\Pi}

\renewcommand{\pmod}[1]{\ (\mathop{\rm mod} #1)}

\DeclareMathOperator{\Cay}{Cay}
\DeclareMathOperator{\Aut}{Aut}

\numberwithin{equation}{section}

\newtheorem{CASE}{Case}
\newtheorem{RomanCase}{Case} 
	\numberwithin{RomanCase}{section}
	
\newtheorem{thm}[CASE]{Theorem}
\newtheorem{prop}[CASE]{Proposition}
\newtheorem{lem}[CASE]{Lemma}
\newtheorem{cor}[CASE]{Corollary}
\newtheorem{mainthm}[CASE]{Theorem}
\newtheorem{FGL}[CASE]{Lemma}
\newtheorem{MarusicMethod}[CASE]{Lemma}
\newtheorem{MarusicMethodCor}[CASE]{Corollary}

\crefformat{CASE}{Case~#2#1#3}
\crefformat{prop}{Proposition~#2#1#3}
\crefformat{thm}{Theorem~#2#1#3}
\crefformat{lem}{Lemma~#2#1#3}
\crefformat{cor}{Corollary~#2#1#3}
\crefformat{mainthm}{Theorem~#2#1#3}
\crefformat{FGL}{the Factor Group Lemma~#2(#1)#3}
\crefformat{MarusicMethod}{Maru\v si\v c's Method~#2(#1)#3}
\crefformat{MarusicMethodCor}{Maru\v si\v c's Method~#2(#1)#3}

\crefname{CASE}{case}{cases}
\crefname{thm}{theorem}{theorems}
\crefname{prop}{proposition}{propositions}
\crefname{lem}{lemma}{lemmas}
\crefname{cor}{corollary}{corollaries}
\crefname{mainthm}{theorem}{theorems}
\crefname{FGL}{lemma}{lemmas}
\crefname{MarusicMethod}{method}{methods}
\crefname{MarusicMethodCor}{method}{methods}

\theoremstyle{definition}
\newtheorem*{ack}{Acknowledgments}
\newtheorem{assump}[CASE]{Assumption}
\newtheorem{defn}[CASE]{Definition}
\newtheorem{rem}[CASE]{Remark}

\theoremstyle{remark}
\newtheorem{notation}[CASE]{Notation}

\crefformat{rem}{Remark~#2#1#3}
\crefformat{assump}{Assumption~#2#1#3}
\crefformat{notation}{Notation~#2#1#3}

\crefname{rem}{remark}{remarks}
\crefname{assump}{assumption}{assumptions}
\crefname{notation}{notation}{notations}

\crefformat{figure}{Figure~#2#1#3}
\crefname{figure}{figure}{figures}
\crefformat{section}{Section~#2#1#3}

\numberwithin{CASE}{section}
\numberwithin{equation}{CASE}

 \newcounter{subcase}
 \newenvironment{subcase}[1][\unskip]{\refstepcounter{subcase}\bf
 \medskip \noindent \hskip\parindent Subcase \thesubcase\ #1. \it}{\unskip\upshape}
\crefformat{subcase}{Subcase~#2#1#3}
\numberwithin{subcase}{CASE}
\numberwithin{subcase}{RomanCase}
\renewcommand{\thesubcase}{\roman{subcase}}
\crefname{subcase}{subcase}{subcases}

\newenvironment{case}[1][\unskip]{\refstepcounter{case}%
 \em
 \medskip \noindent Case \thecase\ #1.\ }{\unskip\upshape}
 \newcommand{\thecase}{\arabic{case}}
 \newcounter{case}
\crefname{case}{case}{cases}

\crefformat{case}{Case~#2#1#3}

\renewenvironment{claim}[1][\unskip]{\em
 \medbreak \indent Claim.\ }{\unskip\upshape}

\makeatletter
\newcommand{\MR}[1]{\href{http://www.ams.org/mathscinet-getitem?mr=#1}{MR~#1}}
\makeatother

\usepackage{color}
\newcommand{\refnote}[1]{%
	\marginpar{%
		\color{blue}
		\vbox to 0pt{\vss
		$\begin{pmatrix} \text{note} \\ \text{\cref{#1}} \end{pmatrix}$%
		\vskip -11pt}}}
\newcommand{\refnotelower}[1]{%
	\marginpar{%
		\color{blue}
		\vbox to 0pt{\vskip 9pt
		$\begin{pmatrix} \text{note} \\ \text{\cref{#1}} \end{pmatrix}$%
		\vss}}}

\theoremstyle{definition}
\newtheorem{aid}[CASE]{}
\newcommand{\oldendaid}{}
\let\oldendaid=\endaid
\renewcommand{\endaid}{\oldendaid\bigskip\hrule width\textwidth \bigbreak}

\renewcommand{\contentsname}{\normalfont\normalsize \hskip\parindent Here is an outline of the paper:} 
\makeatletter
\renewcommand{\l@section}{\@dottedtocline{1}{3.5em}{1em}}
\renewcommand{\l@subsection}{\@dottedtocline{2}{6em}{1.5em}}
\makeatother

\makeatletter
\renewenvironment{frontmatter}
{\thispagestyle{amctitle}%
\setcounter{page}{\@startpage}%
\vskip 10pt%
\centering}
{\vskip 20pt%
\blfootnote{\raggedright \ifnum\@authorcount=1\textit{E-mail address:}\else\textit{E-mail addresses:}\fi ~\@emails}%
}
\renewcommand{\ps@amctitle}{%
  \renewcommand\@oddhead{}
  \let\@evenhead\@oddhead
  \renewcommand\@evenfoot{\hfil}
  \let\@oddfoot\@evenfoot
}
\def\@oddrunninghead{Odd-order Cayley graphs with commutator 
subgroup of order~$pq$ are hamiltonian}
\def\@evenrunninghead{Dave Witte Morris}
\makeatother

\makeatletter
\newcommand{\noprelistbreak}{\smallskip\@nobreaktrue\nopagebreak} 
\makeatother

\begin{document}

\begin{frontmatter}

\begingroup
\mathversion{bold} 
\titledata{Odd-order Cayley graphs with commutator 
\\ subgroup of order~$pq$ are hamiltonian}{}
{(version of \today)} 
\endgroup

\authordata{Dave Witte Morris}
{Department of Mathematics and Computer Science,
University of Lethbridge, Lethbridge, Alberta, T1K~3M4, Canada}
{Dave.Morris@uleth.ca, http://people.uleth.ca/$\sim$dave.morris/}
{}


\begin{abstract}
We show that if $G$ is a nontrivial, finite group of odd order, whose commutator subgroup $[G,G]$ is cyclic of order~$p^\mu q^\nu$, where $p$ and~$q$ are prime, then every connected Cayley graph on~$G$ has a hamiltonian cycle.
\end{abstract}

\end{frontmatter}

\section{Introduction}

It has been conjectured that there is a hamiltonian cycle in every connected Cayley graph on any finite group, but all known results on this problem have very restrictive hypotheses (see \cite{CurranGallian-survey, PakRadoicic-survey, WitteGallian-survey} for surveys). One approach is to assume that the group is close to being abelian, in the sense that its commutator subgroup is small. This is illustrated by the following theorem that was proved in a series of papers by Maru\v si\v c \cite{Marusic-HamCircCay}, Durnberger \cite{Durnberger-semiprod,Durnberger-prime}, and  Keating-Witte \cite{KeatingWitte}:

\begin{thm}[D.\,Maru\v si\v c, E.\,Durnberger, K.\,Keating, and D.\,Witte, 1985] \label{KeatingWitteThm}
If $G$ is a nontrivial, finite group, whose commutator subgroup $[G,G]$ is cyclic of order~$p^\mu$, where $p$~prime and $\mu \in \natural$, then every connected Cayley graph on~$G$ has a hamiltonian cycle.
\end{thm}

Under the additional assumption that $G$ has odd order, we extend this theorem, by allowing the order of $[G,G]$ to be the product of two prime-powers:

\begin{mainthm} \label{MAINTHM}
If $G$ is a nontrivial, finite group of odd order, whose commutator subgroup $[G,G]$ is cyclic of order~$p^\mu q^\nu$, where $p$ and~$q$ are prime, and $\mu,\nu \in \natural$, then every connected Cayley graph on~$G$ has a hamiltonian cycle.
\end{mainthm}

\begin{rem}
Of course, we would like to prove the conclusion of \cref{MAINTHM} without the assumption that $|G|$ is odd, or with a weaker assumption on the order of $[G,G]$.
\end{rem}

If $\mu,\nu \le 1$, then there is no need to assume that $[G,G]$ is cyclic:

\begin{cor} \label{G'=pq}
If $G$ is a nontrivial, finite group of odd order, whose commutator subgroup $[G,G]$ has order~$pq$, where $p$ and~$q$ are distinct primes, then every connected Cayley graph on~$G$ has a hamiltonian cycle.
\end{cor}

This yields the following contribution to the ongoing search \cite{M2Slovenian-LowOrder} for hamiltonian cycles in Cayley graphs on groups whose order has few prime factors:

\begin{cor} \label{9pq}
If $p$ and~$q$ are distinct primes, then every connected Cayley graph of order $9pq$ has a hamiltonian cycle.
\end{cor}

\tableofcontents 

\begin{ack}
I thank D.\,Maru\v si\v c for suggesting this research problem. I also thank him, K.\,Kutnar, and other members of the Faculty of Mathematics, Natural Sciences, and Information Technologies of the University of Primorska (Koper, Slovenia), for their excellent hospitality that supported the early stages of this work. 
\end{ack}

\section{Preliminaries}

\subsection{Assumptions, definitions, and notation}

\begin{assump} \ 
\noprelistbreak
	\begin{enumerate}
	\item $G$ is always a finite group.
	\item $S$ is a generating set for~$G$.
	\end{enumerate}
\end{assump}

\begin{defn}
The \emph{Cayley graph} $\Cay(G;S)$ is the graph whose vertex set is~$G$, with an edge from $g$ to~$gs$ and an edge from~$g$ to~$gs^{-1}$, for every $g \in G$ and $s \in S$.
\end{defn}

\begin{notation} \ 
\noprelistbreak
	\begin{itemize}
	\item We let $G' = [G,G]$ and $\quot{G} = G/G'$. Also, for $g \in G$, we let $\quot{g} = gG'$ be the image of~$g$ in~$\quot{G}$.
	\item For $g,h \in G$, we let $g^h = h^{-1} g h$ and $[g,h] = g^{-1} h^{-1} g h$.
	\item If $H$ is an abelian subgroup of~$G$ and $k \in \integer$, we let
	$$ H^k = \{\, h^k \mid h \in H \,\} .$$
This is a subgroup of~$H$ (because $H$ is abelian).
\end{itemize}
\end{notation}

\begin{notation}
For $g \in G$ and $s_1,\ldots,s_n \in S \cup S^{-1}$, we use $[g](s_1,\ldots,s_n)$ to denote the walk in $\Cay(G;S)$ that visits (in order), the vertices
	$$ g, \, gs_1, \ gs_1 s_2, \, g s_1 s_2 s_3, \ \ldots, \, gs_1s_2\cdots s_n .$$
We often write $(s_1,\ldots,s_n)$ for $[e](s_1,\ldots,s_n)$.
\end{notation}

\begin{defn}
Suppose
\noprelistbreak
 	\begin{itemize}
	\item $N$ is a normal subgroup of~$G$,
	and
	\item $C = (s_i)_{i=1}^n$ is a hamiltonian cycle in $\Cay(G/N;  S)$.
	\end{itemize}
The \emph{voltage} of~$C$ is $\prod_{i=1}^n s_i$. This is an element of~$N$, and it may be denoted $\voltage C$.
\end{defn}

\begin{rem} \label{gVoltage}
If $C = [g](s_1,\ldots,s_n)$, then
	$ \prod_{i=1}^n s_i  = (\voltage C)^g $.
\end{rem}

\begin{proof}
There is some $\ell$ with $\prod_{i=1}^\ell s_i \in g^{-1} N$. Then 
	$$ C = (s_{\ell+1}, s_{\ell+2}, \ldots, s_n, s_1, s_2, \ldots, s_\ell ) ,$$
so
	\begin{align*}
	(\voltage C)^g
	&= g^{-1} (s_{\ell+1} s_{\ell+2} \cdots s_n \, s_1 s_2 \cdots s_\ell ) g
	\\&= 
\left( \prod_{i=1}^\ell s_i \right)  \left( \prod_{i=\ell+1}^n s_i \right)  \left( \prod_{i=1}^\ell s_i \right)  \left( \prod_{i=1}^\ell s_i \right) ^{-1}
	\\&= 	\prod_{i=1}^n s_i
	.\qedhere \end{align*}
\end{proof}

\subsection{Factor Group Lemma and Maru\v si\v c's Method}

\begin{FGL}[``Factor Group Lemma'' {\cite[\S2.2]{WitteGallian-survey}}] \label{FGL}
Suppose
\noprelistbreak
 \begin{itemize}
 \item $N$ is a cyclic, normal subgroup of~$G$,
 \item $(s_i)_{i=1}^m$ is a hamiltonian cycle in $\Cay(G/N;S)$,
 and
 \item the product $s_1s_2\cdots s_m$ generates~$N$.
 \end{itemize}
 Then $(s_1,s_2,\ldots,s_m)^{|N|}$ is a hamiltonian cycle in $\Cay(G;S)$.
 \end{FGL}
 
 The following simple observation allows us to assume $|N|$ is square-free whenever we apply \cref{FGL}.
 
 \begin{lem}[{\cite[Lem.~3.2]{KeatingWitte}}] \label{FreeLunch}
Suppose
\noprelistbreak
 \begin{itemize}
 \item $N$ is a cyclic, normal subgroup of~$G$,
 \item $\underline{N} = N/\Phi$ is the maximal quotient of~$N$ that has square-free order (so $\Phi$ is the ``Frattini subgroup'' of~$N$),
 \item $\underline{G} = G/\Phi$,
 \item $(s_1,s_2,\ldots,s_m)$ is a hamiltonian cycle in $\Cay(\underline{G}/\underline{N};S)$,
 and
 \item the product $\underline{s_1} \, \underline{s_2} \cdots \underline{s_m}$ generates~$\underline{N}$.
 \end{itemize}
 Then $s_1s_2\cdots s_m$ generates~$N$, so  $(s_1,s_2,\ldots,s_m)^{|N|}$ is a hamiltonian cycle in $\Cay(G;S)$.
 \end{lem}
 
 \begin{rem}[{cf.\ \cite[Thm.~5.1.1]{Gorenstein-FinGrps}}] \label{PQStillMin}
 When applying \cref{FreeLunch}, it is sometimes helpful to know that if
\noprelistbreak
 \begin{itemize}
 \item $N$, $\underline{N} = N/\Phi$, and $\underline{G} = G/\Phi$ are as in \cref{FreeLunch},
 and
 \item $S$ is a minimal generating set of~$G$.
 \end{itemize}
 Then $\underline{S}$ is a minimal generating set of~$\underline{G}$.
 \end{rem}

 \begin{MarusicMethod}[``Maru\v si\v c's Method'' {\cite{Marusic-HamCircCay}, cf.\ \cite[Lem.~3.1]{KeatingWitte}}] \label{MarusicMethod}
 Suppose 	
 \noprelistbreak
	\begin{itemize}
	\item $S_0 \subseteq S$,
	\item $\langle S_0 \rangle$ contains~$G'$,
	\item there are hamiltonian cycles $C_1,\ldots,C_r$ in $\Cay(\langle S_0 \rangle / G' ; S_0)$ that all have an oriented edge in common,
	and
	\item[\Large$*$\hskip-1.5pt]
	for every $\gamma \in G'$, there is some~$i$, such that $\bigl\langle \gamma \cdot \voltage C_i \bigr\rangle = G'$.
	\end{itemize}
Then there is a hamiltonian cycle in $\Cay(G/G';S)$ whose voltage generates~$G'$. Hence, \cref{FGL} provides a hamiltonian cycle in $\Cay(G;S)$.
 \end{MarusicMethod}

\begin{MarusicMethodCor} \label{MarusicMethod34}
Assume $G' = \integer_p \times \integer_q$, where $p$ and~$q$ are distinct primes. Then, in the situation of \cref{MarusicMethod}, the final condition ({\Large$*$}) 
can be replaced with either of the following:
\noprelistbreak
	\begin{enumerate}
	
	\item \label{MarusicMethod34-3}
	$r = 3$, and $\bigl\langle (\voltage C_i)^{-1} (\voltage C_j) \bigr\rangle = G'$ whenever $1 \le i < j \le 3$.
	
	\item \label{MarusicMethod34-4}
	$r = 4$, and 
	\noprelistbreak
		\begin{itemize}
		\item $\bigl\langle (\voltage C_1)^{-1} (\voltage C_2) \bigr\rangle$ contains~$\integer_p$, 
		and
		\item $\bigl\langle (\voltage C_1)^{-1} (\voltage C_3) \bigr\rangle = \bigl\langle (\voltage C_2)^{-1} (\voltage C_4) \bigr\rangle = \integer_q$.
		\end{itemize}
	
	\end{enumerate}
\end{MarusicMethodCor}

\begin{proof}
Let $\gamma \in G'$.

\pref{MarusicMethod34-3} Consider the three elements $\gamma \cdot \voltage C_1$, $\gamma \cdot \voltage C_2$, and~$\gamma \cdot \voltage C_3$ of $\integer_p \times \integer_q$. By assumption, no two have the same projection to~$\integer_p$, so only one of them can have trivial projection. Similarly for the projection to~$\integer_q$. Therefore, there is some~$i$, such that $\gamma \cdot \voltage C_i$ projects nontrivially to both $\integer_p$ and~$\integer_q$. Therefore $\langle \gamma \cdot \voltage C_i \rangle = G'$.

\pref{MarusicMethod34-4} There is some $i \in \{1,2\}$, such that $\gamma \cdot \voltage C_i$ projects nontrivially to~$\integer_p$. We may assume the projection of $\gamma \cdot \voltage C_i$ to~$\integer_q$ is trivial (otherwise, we have $\langle \gamma \cdot \voltage C_i \rangle = G'$, as desired). Then $\gamma \cdot \voltage C_{i+2}$ has the same (nontrivial) projection to~$\integer_p$, but has a different (hence, nontrivial) projection to~$\integer_q$. So $\langle \gamma \cdot \voltage C_{i+2} \rangle = G'$.
\end{proof}

\subsection{Some known results}

We recall a few results that provide hamiltonian cycles in $\Cay(G;S)$ under certain assumptions.

\begin{thm}[{Witte \cite{Witte-pgrp}}] \label{pgrp}
If $|G| = p^\mu$, where $p$~is prime and $\mu > 0$, then every connected Cayley digraph on~$G$ has a directed hamiltonian cycle.
\end{thm}

\begin{thm}[{Ghaderpour-Morris \cite{GhaderpourMorris-Nilpotent}}] \label{GhaderpourMorrisNilpotent}
If $G$ is a nontrivial, nilpotent, finite group, and the commutator subgroup of~$G$ is cyclic, then every connected Cayley graph on~$G$ has a hamiltonian cycle.
\end{thm}

\begin{thm}[{Ghaderpour-Morris \cite{GhaderpourMorris-27p}}] \label{27p}
If $|G| = 27p$, where $p$~is prime, then every connected Cayley graph on~$G$ has a hamiltonian cycle.
\end{thm}

\begin{cor}[{of proof}] \label{G=27palpha}
If $G$ is a finite group, such that $|G/G'| = 9$ and $G'$ is cyclic of order $p^\mu \cdot 3^\nu$, where $p \ge 5$~is prime, then every connected Cayley graph on~$G$ has a hamiltonian cycle.
\end{cor}

\begin{proof}
Let $\underline{G} = G/(G')^{3p}$. Then $|\underline{G}| = 27p$ and $|\underline{G}'| = 3p$, so the proof of \cite[Props.~3.4 and~3.6]{GhaderpourMorris-27p} provides a hamiltonian cycle in $\Cay \bigl( \underline{G} /  \underline{G}'; S \bigr)$ whose voltage generates~$\underline{G}'$. Then \cref{FreeLunch} provides a hamiltonian cycle in $\Cay(G;S)$.
\end{proof}

\begin{thm}[{Alspach \cite[Thm.~3.7]{Alspach-lifting}}] \label{GenNormal}
Suppose
\noprelistbreak
	\begin{itemize}
	\item $s \in S$,
	\item $\langle s \rangle \normal G$,
	\item $|G / \langle s \rangle|$ is odd, 
	and
	\item there is a hamiltonian cycle in $\Cay\bigl( G/\langle s \rangle ; S \bigr)$.
	\end{itemize}
Then there is a hamiltonian cycle in $\Cay(G;S)$.
\end{thm}

This has the following immediate consequence, since every subgroup of a cyclic, normal subgroup is normal:

\begin{cor} \label{GenInG'}
Suppose 
\noprelistbreak
	\begin{itemize}
	\item $G'$ is cyclic,
	\item $s \in S \cap G'$,
	\item $|G/\langle s \rangle|$ is odd,
	and
	\item there is a hamiltonian cycle in $\Cay(G/\langle s \rangle; S)$.
	\end{itemize}
Then there is a hamiltonian cycle in $\Cay(G;S)$.
\end{cor}

\subsection{Group theoretic preliminaries}

We recall a few elementary facts about finite groups.

\begin{lem}[{\cite[3.11]
{GhaderpourMorris-Nilpotent}}] 
\label{Cent->Divides}
Suppose 
\noprelistbreak
	\begin{itemize}
	\item $\langle a,b \rangle = G$, 
	\item $G'$ is cyclic of square-free order, 	and 
	\item $G' \subseteq Z(G)$. 
	\end{itemize}
Then $|[a,b]|$ is a divisor of both $\langle \quot{a} \rangle$ and $|\quot{G}/\langle \quot{a} \rangle|$.
\end{lem}

\begin{lem}[{\cite[Lem.~3.12]{GhaderpourMorris-Nilpotent}}] \label{G'=ab}
If $G = \langle a,b \rangle$, and $G'$ is cyclic, then $G' = \langle [a,b] \rangle$.
\end{lem}

\begin{cor} \label{aCents->proper}
Suppose 
\noprelistbreak
	\begin{itemize}
	\item $\langle a, G' \rangle = G$, 
	and
	\item $G'$ is cyclic of square-free order. 
	\end{itemize}
Then $a$ does not centralize any nontrivial subgroup of~$G'$.
\end{cor}

\begin{proof}
Let $\gamma$ be a generator of the cyclic group~$G'$, and let $\underline{G} = G / \langle [a, \gamma] \rangle$, so $\underline{a}$ centralizes $\underline{\gamma}$. Then $\underline{G}'  = \langle \underline{\gamma} \rangle \subseteq Z( \underline{G} )$, so \cref{Cent->Divides} tells us that $|\underline{G}'| = | [\underline{a}, \underline{\gamma}]|$ is a divisor of $|\quot{G}/\langle \quot{a} \rangle| = 1$. This means $\underline{G}$ is abelian, so $\langle [a, \gamma] \rangle = G' = \langle \gamma \rangle$. This implies that $a$ does not centralize any nontrivial power of~$\gamma$. In other words, $a$ does not centralize any nontrivial subgroup of~$G'$.
\end{proof}

\begin{lem} \label{27grp}
Suppose 
\noprelistbreak
	\begin{itemize}
	\item $G' = \integer_{3^\mu}$ is cyclic of order~$3^\mu$, for some $\mu \in \natural$,
	and
	\item $G/(G')^3$ is a nonabelian group of order~$27$.
	\end{itemize}
Then
\noprelistbreak
	\begin{enumerate}
	\item \label{27grp-subgrp}
	the elements of order~$3$ (together with~$e$) form a subgroup of~$G$,
	\item \label{27grp-mu=1}
	$\mu = 1$ (so $|G| = 27$),
	and
	\item \label{27grp-a3b3}
	$ (ab)^3 = a^3 b^3$ for all $a,b \in G$.
	\end{enumerate}
\end{lem}

\begin{proof}
Note that $|G| = 3^{\mu+2}$, so $G$ is a $3$-group.
Since $G'$ is cyclic (and $3$ is odd), it is not difficult to show\refnotelower{27grp-regularPf}
	\begin{align} \label{27grp-regular}
	\text{$(ab)^3 \in a^3 b^3 (G')^3$, for all $a,b \in G$.}
	\end{align}
(This is a special case of \cite[Satz III.10.2(c), p.~322]{Huppert}.)

\pref{27grp-subgrp} This is immediate from \pref{27grp-regular}. (This is a special case of \cite[Satz III.10.6(a), p.~326]{Huppert}.)

\pref{27grp-mu=1} 
Since $G/G' \iso \integer_3 \times \integer_3$, there is a $2$-element generating set $\{a,b\}$ of~$G$. (In fact, every minimal generating set has exactly two elements \cite[3.15, p.~273]{Huppert}.)
Since $a^3, b^3 \in G'$, we see from \pref{27grp-regular} that we may assume $b^3 \in (G')^3$ (by replacing $b$ with $ba$ or~$ba^{-1}$, if necessary). Furthermore, by modding out $(G')^9$, there is no harm in assuming $\mu \le 2$, so $(G')^3 \subseteq Z(G)$. Therefore $[a, b^3] = e$, so \cite[Satz 10.6(b), p.~326]{Huppert} tells us that $[a,b]^3 = e$. Since $\langle [a,b] \rangle = G'$ \csee{G'=ab}, this implies $\mu = 1$.

\pref{27grp-a3b3} Since $\mu = 1$, we have $(G')^3 = \{e\}$, so this is immediate from \pref{27grp-regular}.
\end{proof}

\subsection{\texorpdfstring{Proofs of \Cref{G'=pq,9pq}}{Proofs of the corollaries}}

\begin{proof}[\bf Proof of \cref{G'=pq}]
Assume, without loss of generality, that $p < q$. Then Sylow's Theorem implies that $G'$ has a unique Sylow $q$-subgroup~$Q$, so $Q \normal G$.  Therefore $G$ acts on~$Q$ by conjugation. Since $Q \iso \integer_q$, we know that the automorphism group of $Q$ is abelian (more precisely, it is cyclic of order $q-1$), so this implies that $G'$ centralizes~$Q$. So $Q \subseteq Z(G')$. Since $G'/Q$ is cyclic (indeed, it is of prime order, namely,~$p$), this implies that $G'$ is abelian. Since $p \neq q$, we know that every abelian group of order~$pq$ is cyclic, so we conclude that $G'$ is cyclic. Therefore \cref{MAINTHM} applies.
\end{proof}

\begin{proof}[\bf Proof of \cref{9pq}]
Assume $|G| = 9pq$. We may assume $p$ and~$q$ are odd, for otherwise $|G|$ is of the form $18p$, so \cite[Prop.~9.1]{M2Slovenian-LowOrder} applies. Therefore $|G|$ is odd, so it suffices to show $|G'|$ is a divisor of $pq$, for then \cref{G'=pq} (or \cref{KeatingWitteThm}) applies.

Note that we may assume $3 \notin \{p,q\}$, for otherwise $|G|$ is of the form $27p$, so \cref{27p} applies. Therefore, neither $|\!\Aut(\integer_9)| = 6$ nor $|\!\Aut(\integer_3 \times \integer_3)| = 48$ is divisible by either~$p$ or~$q$, so Burnside's Transfer Theorem \cite[Thm.~7.4.3, p.~252]{Gorenstein-FinGrps} implies that $G$ has a normal subgroup~$N$ of order~$pq$. 
Since $|G/N| = 9$, and every group of order~$9$ is abelian, we know that $G' \subseteq N$, so  $|G'|$ is a divisor of $|N| = pq$, as desired.
\end{proof}

Let us also record the fact that almost all cases of \cref{MAINTHM} will be proved by using \cref{MarusicMethod}:

\begin{thm}
Assume
	\begin{itemize}
	\item $S$ is a minimal generating set for a nontrivial, finite group~$G$ of odd order,
	\item $G'$ is cyclic of order~$p^\mu q^\nu$, where $p$ and~$q$ are prime, and $\mu,\nu \in \natural$, 
	\item for all $s \in S$, we have $s \notin G'$ and $G' \not\subseteq \langle s \rangle$,
	\item $G/(G')^3$ is \textbf{not} the nonabelian group of order~$27$ and exponent~$3$,
	and
	\item either $G/G' \not\iso \integer_3 \times \integer_3$, or $\#S \neq 2$.
	\end{itemize}
Then, for every $\gamma \in G$', there exists a hamiltonian cycle~$C$ in $\Cay \bigl( G/G' ; S \bigr)$, such that $\gamma \voltage C$ generates~$G'$.
\end{thm}

\section{The usual application of Maru\v si\v c's Method}

Applying \cref{MarusicMethod} requires the existence of more than one hamiltonian cycle in a quotient of $\Cay(G;S)$. In practice, one usually starts with a single hamiltonian cycle and modifies it in various ways to obtain the others that are needed. The following result describes a modification that will be used repeatedly in the proof of \cref{MAINTHM}.

\begin{lem}[cf.\ Durnberger \cite{Durnberger-semiprod} and Maru\v si\v c \cite{Marusic-HamCircCay}] \label{StandardAlteration}
Assume:
\noprelistbreak
	\begin{itemize}
	\item $C_0$ is an oriented hamiltonian cycle in $\Cay(\quot{G};S)$,
	\item $a,b \in S^{\pm1}$, $g \in G$, and $m \in \integer^+$, 
	\item $C_0$ contains:
	\noprelistbreak
		\begin{itemize}
		\item the oriented path $[ga^{-(m+1)}](a^m, b, a^{-m})$, 
		and
		\item either the oriented edge $[g](b)$ or the oriented edge $[gb](b^{-1})$. 
		\end{itemize}
	\end{itemize}
Then there are hamiltonian cycles $C_0$, $C_1$, \dots, $C_m$ in $\Cay(\quot{G};S)$, such that
	$$ \left( \bigl( \voltage C_0 \bigr)^{-1} \bigl( \voltage C_k \bigr) \right)^g 
	= \begin{cases}
	[a^k, b^{-1}] \, [a^k, b^{-1}]^a
		&\text{if $C_0$ contains $[g](b)$}, \\
	[b^{-1}, a^k] \, [a^k, b^{-1}]^a
		&\text{if $C_0$ contains $[gb](b^{-1})$}
	. \end{cases}$$
\end{lem}
\begin{figure}
\begin{center}
	\vbox{
	\includegraphics{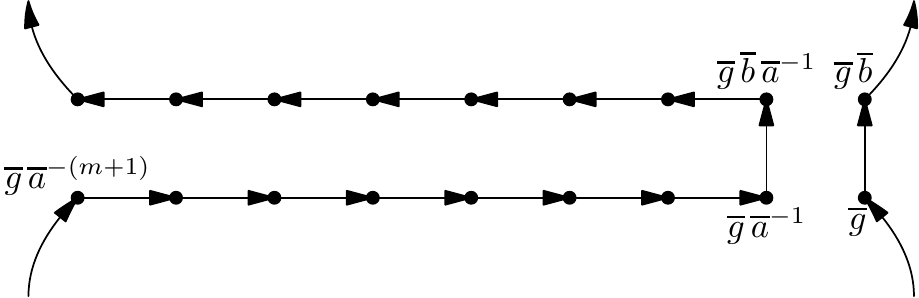}
	\vskip0.25in
	\includegraphics{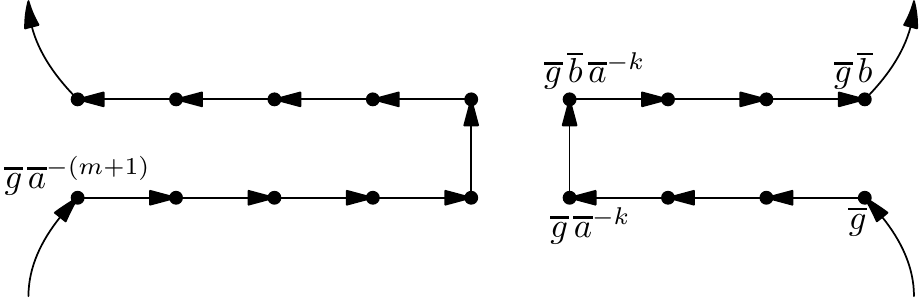}
	}
\end{center}
\caption{A portion of the hamiltonian cycles~$C_0$ (top) and~$C_k$ (bottom).}
\label{standardFig}
\end{figure}

\begin{proof}
Note that $[ga^{-(m+1)}](a^m, b, a^{-m})$ contains the subpath $[ga^{-(k+1)}](a^k, b, a^{-k})$ for $0 \le k \le m$.

\setcounter{case}{0}

\begin{case}
Assume that $C_0$ contains $[g](b)$. 
\end{case}
Construct $C_k$ by:
\noprelistbreak
	\begin{itemize}
	\item replacing the oriented edge $[g](b)$ with the oriented path $[g](a^{-k}, b, a^{k})$,
	and
	\item replacing the oriented path $[ga^{-(k+1)}](a^k, b, a^{-k})$ with the oriented edge $[ga^{-(k+1)}](b)$
	\end{itemize}
\csee{standardFig}.

To calculate the voltage of~$C_k$, write $C_0 = [g](s_1,\ldots,s_n)$. There is some~$\ell$ with $\quot{s_1}\cdots \quot{s_\ell} = \quot{a}^{-1}$, so
	$$ C_k = [g]\bigl( a^{-k}, b, a^{k}, \,  ( s_i)_{i=2}^{\ell-k} , \, b , (s_i)_{i=\ell+k+2}^n \bigr) .$$
For convenience, let
	$$ h =  \prod_{i=\ell+1}^n s_i \equiv  \left( \prod_{i=1}^\ell s_i \right)^{-1} \equiv a \pmod{G'}.$$
Then, from \cref{gVoltage} (and the fact that $G'$ is commutative), we have
	\begin{align*}
	\bigl(\voltage C_k \bigr)^g
	&= (a^{-k} b a^{k}) \left( \prod_{i=2}^{\ell-k} s_i \right) b  \left( \prod_{i=\ell+k+2}^n s_i \right) 
	\\&= (a^{-k} b a^{k}b^{-1}) \left( \prod_{i=1}^{\ell} s_i \right) a^{-k} b a^k b^{-1} \left( \prod_{i=\ell+1}^n s_i \right) 
	\\&= [a^{k}, b^{-1}] \cdot \left( \prod_{i=1}^{\ell} s_i \right) \left( \prod_{i=\ell+1}^n s_i \right) \cdot [a^k , b^{-1}]^h
	\\&= [a^{k}, b^{-1}] \cdot \bigl( \voltage C_0 \bigr)^g \cdot [a^k , b^{-1}]^a
	\\&= \bigl( \voltage C_0 \bigr)^g \cdot [a^{k}, b^{-1}] \,  [a^k , b^{-1}]^a
	. \end{align*}

\begin{case}
Assume that $C_0$ contains $[gb](b^{-1})$. 
\end{case}
This is similar.
Construct $C_k$ by:
\noprelistbreak
	\begin{itemize}
	\item replacing the oriented edge $[gb](b^{-1})$ with the oriented path $[gb](a^{-k}, b^{-1}, a^{k})$,
	and
	\item replacing the oriented path $[ga^{-(k+1)}](a^k, b, a^{-k})$ with the oriented edge $[ga^{-(k+1)}](b)$.
	\end{itemize}
(See \cref{standardFig}, but reverse the orientation of the paths in the right half of the figure.)

To calculate the voltage of~$C_k$, write $C_0 = [gb](s_1,\ldots,s_n)$. There is some~$\ell$ with $\quot{s_1}\cdots \quot{s_\ell} = \quot{ab}^{-1}$, so
	$$ C_k = [gb]\bigl( a^{-k}, b^{-1}, a^{k}, \,  ( s_i)_{i=2}^{\ell-k} , \, b , (s_i)_{i=\ell+k+2}^n \bigr) .$$
For convenience, let
	$$ h =  \prod_{i=\ell+1}^n s_i \equiv  \left( \prod_{i=1}^\ell s_i \right)^{-1} \equiv ab \pmod{G'}.$$
Then
	\begin{align*}
	\bigl(\voltage C_k \bigr)^{gb}
	&= (a^{-k} b^{-1} a^{k}) \left( \prod_{i=2}^{\ell-k} s_i \right) b  \left( \prod_{i=\ell+k+2}^n s_i \right) 
	\\&= (a^{-k} b^{-1} a^{k}b) \left( \prod_{i=1}^{\ell} s_i \right) a^{-k} b a^k b^{-1} \left( \prod_{i=\ell+1}^n s_i \right) 
	\\&= b^{-1}(b a^{-k} b^{-1} a^{k}) b \cdot \left( \prod_{i=1}^{\ell} s_i \right) \left( \prod_{i=\ell+1}^n s_i \right) \cdot [a^k , b^{-1}]^h
	\\&= [b^{-1}, a^{k}]^b \cdot \bigl( \voltage C_0 \bigr)^{gb} \cdot [a^k , b^{-1}]^{ab}
	\\&= \bigl( \voltage C_0 \bigr)^{gb} \cdot [b^{-1}, a^{k}]^b \,  [a^k , b^{-1}]^{ab}
	. \qedhere \end{align*}
\end{proof}

\begin{rem} \label{StandardAlterationIsG'}
In the situation of \cref{StandardAlteration}, we have $\left\langle \bigl( \voltage C_0 \bigr)^{-1} \bigl( \voltage C_k \bigr) \right\rangle = \langle [a^k, b^{-1}] \rangle$ if either
	\begin{enumerate}
	\item \label{StandardAlterationIsG'-NoInvert}
	$C_0$ contains $[g](b)$ and $a$ does not invert any nontrivial element of $\langle [a^k, b^{-1}] \rangle$,
	or
	\item  \label{StandardAlterationIsG'-NoCent}
	$C_0$ contains $[gb](b^{-1})$ and $a$ does not centralize any nontrivial element of $\langle [a^k, b^{-1}] \rangle$.
	\end{enumerate}
Note that if $|G|$ is odd, then the hypothesis on~$a$ in \pref{StandardAlterationIsG'-NoInvert} is automatically satisfied (because no element of odd order can ever invert a nontrivial element).
\end{rem}

\begin{cor}[cf.\ {\cite[Case iv]{Durnberger-prime} and \cite[Case 4.3]{KeatingWitte}}] \label{KW43}
Assume
\noprelistbreak
	\begin{itemize}
	\item $a \in S$ with $\langle \quot{a} \rangle \neq \quot{G}$,
	\item $(s_i)_{i=1}^d$ is a hamiltonian cycle in $\Cay \bigl( \quot{G}/\langle \quot{a} \rangle ; S \bigr)$,
	\item $a^r \prod_{i=1}^d s_i \in G'$, with $0 \le r \le |\quot{a}| - 2$,
	and
	\item $0 \le k \le |\quot{a}| - 3$.
	\end{itemize}
Then the walk
	\begin{align*} C_k = \bigl( a^k, s_1, & a^{-(k+1)}, 
	 ( s_{2i}, a^{|\quot{a}|-2}, s_{2i+1}, a^{-(|\quot{a}|-2)} )_{i=1}^{(d-3)/2}, 
	\\& s_{d-1}, a^r, s_d, 
	a^{-(|\quot{a}|-k-2)}, s_1, a^{|\quot{a}|-k-3}, (s_i)_{i=2}^{d-1}, 
	a^{-(|\quot{a}|-r-2)}, s_d \bigr)
	 \end{align*}
is a hamiltonian cycle in $\Cay( \quot{G}; S)$ \csee{KW43Fig}, and we have
	$$ \voltage C_k = (\voltage C_0) [a^{-k},  s_1^{-1}]  [a^{-k} , s_1^{-1}] ^{a^{-1}} .$$
\end{cor}

\begin{proof}
$C_0$ contains the oriented edge $(s_1)$ and the oriented path $[a^{|\quot{a}|-2}]( a^{-(|\quot{a}|-3)}, s_1, a^{|\quot{a}|-3})$, so we may apply \cref{StandardAlteration} with $g = e$, $b = s_1$, and $a^{-1}$ in the role of~$a$. 
\end{proof}

\begin{figure}[t]
\begin{center}
\includegraphics{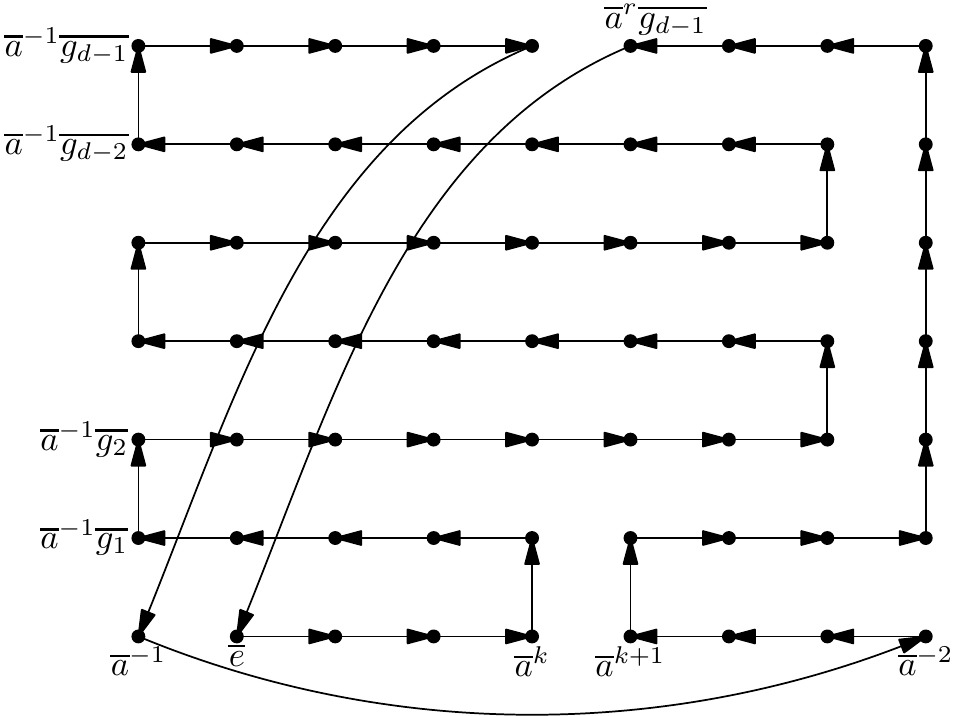}
\caption{A hamiltonian cycle $C_k$ in $\Cay \bigl( \quot{G}; S \bigr)$,
where $g_j = \prod_{i=1}^j s_i$.}
\label{KW43Fig}
\end{center}
\end{figure}

\section{Other applications of Maru\v si\v c's Method}

Here are some other situations in which we can apply \cref{MarusicMethod}.

\begin{thm}[{\cite[\S4 and \S5]{KeatingWitte}}] \label{G'=p}
Suppose 
\noprelistbreak
	\begin{itemize}
	\item $|G|$ is odd,
	\item $G' = \integer_{p^\mu}$ is cyclic of prime-power order,
	\item $S$ is a generating set of~$G$,
	\item $S \cap G' = \emptyset$,
	and
	\item $G$ is \textbf{not} the nonabelian group of order~$27$ with exponent~$3$.
	\end{itemize}
Then there exist hamiltonian cycles $C_1$ and~$C_2$ in $\Cay(G / G' ; S)$ that have an oriented edge in common, such that $(\voltage C_1)^{-1} (\voltage C_2)$ generates~$G'$.
\end{thm}

\begin{proof}
\Cref{FreeLunch} allows us to assume $|G'| = p$. Then the desired conclusion is implicit in \cite[\S4 and \S5]{KeatingWitte} unless $|G/G'| \iso \integer_3 \times \integer_3$ and $p = 3$.  

Therefore $G/(G')^3$ is a nonabelian group of order~$27$, so \fullcref{27grp}{mu=1} tells us $|G| = 27$.
By assumption, the exponent of~$G$ is greater than~$3$, so we conclude from \fullcref{27grp}{subgrp} that
$S$ contains an element~$b$ with $|b| \ge 9$. We may assume $S$ is minimal, so $\#S = 2$; write $S = \{a,b\}$. Then we have the following two hamiltonian cycles in $\Cay(\quot{G}; S)$:
	\begin{align*}
	\text{$C_1 = (a^2, b)^3$ \ and \ $C_2 = (a^2, b^{-1})^3$} 
	. \end{align*}
Since \fullcref{27grp}{a3b3} tells us $(xy)^3 = x^3 y^3 $ for all $x,y \in G$, and we have $x^3 \in G' = Z(G)$ for all $x \in G$, we see that
	$$ (\voltage C_1)^{-1} (\voltage C_2)
	= \bigl( (a^2 b)^3 \bigr)^{-1} \bigl( a^2 b^{-1} \bigr)^3
	= \bigl( (a^2)^3 b^3 \bigr)^{-1} \bigl( (a^2)^3 (b^{-1})^3 \bigr)
	= b^{-6}
	\neq e ,$$
since $|b| \ge 9$.
\end{proof}

We will use the following version of this result in \cref{Not27} of \cref{G'NotIn<abc>}.

\begin{cor}[{of proof}] \label{G'=p/Z}
Suppose 
\noprelistbreak
	\begin{itemize}
	\item $|G|$ is odd,
	\item $G' = \integer_{p}$ has prime order,
	\item $Z$ is a subgroup of $Z(G)$,
	\item $S \cap G'Z = \emptyset$, 
	and
	\item $G$ is not nilpotent.
	\end{itemize}
Then there exist hamiltonian cycles $C_1$ and~$C_2$ in $\Cay(G / (G'Z) ; S)$ that  have an oriented edge in common,
such that $\langle (\voltage C_1)^{-1} (\voltage C_2) \rangle = G'$.
\end{cor}

\begin{proof}
Choose $a,b \in S$ with $[a,b] \neq e$. Since $G$ is not nilpotent, we may assume $a$ does not centralize~$G'$. Furthermore, since we are using \cref{MarusicMethod}, there is no harm in assuming $S = \{a,b\}$. 

If $b \notin \langle a, G', Z \rangle$, then the proof of \cite[Case~5.3]{KeatingWitte} provides two hamiltonian cycles $C_1 = (s_i)_{i=1}^n$ and $C_2 = (t_i)_{i=1}^n$ in $\Cay \bigl( G/(G'Z) ; a,b\bigr)$, such that $\voltage C_1 \neq \voltage C_2$ (and the two cycles have an oriented edge in common). From the construction, it is clear that $(s_i)_{i=1}^n$ is a permutation of $(t_i)_{i=1}^n$, so $(\voltage C_1)^{-1} (\voltage C_2) \in G'$.

We may now assume $b \in \langle a, G', Z \rangle$. Then, letting $n =  |G : \langle a, G', Z \rangle|$, there is some~$i$, such that $b^i \in a^i G' Z$ and $0 < i < n$. Therefore, we have the following two hamiltonian cycles in  $\Cay \bigl( G/(G'Z) ; S \bigr)$ that both contain the oriented edge~$(b)$:
	\begin{align*}
	 C_1 &= (b, a^{-(i-1)}, b, a^{n-i-1}) , \\
	 C_2 &= (b, a^{n-i-1}, b, a^{-(i-1)}) = [a] C_0
	 . \end{align*}
The sequence of edges in~$C_2$ is a permutation of the sequence of edges in~$C_1$, so  $(\voltage C_1)^{-1} (\voltage C_2) \in G'$. Also, since $a$ does not centralize~$G'$, it is not difficult to see that $(\voltage C_1)^{-1} (\voltage C_2)$ is nontrivial,\refnote{G'=p/ZNontrivial} and therefore generates~$G'$.
\end{proof}

\begin{lem} \label{aG'=bG'}
Assume
\noprelistbreak
	\begin{itemize}
	\item $G' = \integer_{p^\mu} \times \integer_{q^\nu}$, where $p$ and~$q$ are prime,
	\item $S \cap G' = \emptyset$,
	\item there exist $a,b \in S \cup S^{-1}$, with $a \neq b$, such that $aG' = bG'$,
	\item the generating set~$S$ is minimal,
	and
	\item $|G|$ is odd.
	\end{itemize}
Then there is a hamiltonian cycle in $\Cay(G;S)$.
\end{lem}

\begin{proof}
Write $b = a \gamma$, with $\gamma \in G'$. 

\setcounter{case}{0}

\begin{case} \label{aG'=bG'-G'Case}
Assume $\langle \gamma \rangle = G'$.
\end{case}
We apply \cref{MarusicMethod}, so \cref{FreeLunch} allows us to assume $G' = \integer_p \times \integer_q$.
Since $| \quot{a} | \ge 3$, it is easy to find an oriented hamiltonian cycle~$C_0$ in $\Cay \bigl( \quot{G} ; S \bigr)$ that has (at least)~$2$ oriented edges~$\alpha_1$ and~$\alpha_2$ that are labeled~$a$. We construct two more hamiltonian cycles $C_1$ and~$C_2$ by replacing one or both of $\alpha_1$ and~$\alpha_2$ with a $b$-edge. (Replace one $a$-edge to obtain $C_1$; replace both to obtain~$C_2$.) Then there are conjugates $\gamma_1$ and~$\gamma_2$ of~$\gamma$, such that\refnotelower{aG'=bG'Voltage} 
	\begin{align*}
	 (\voltage C_0)^{-1}(\voltage C_1)  =  \gamma_1 , 
	 \quad
	 (\voltage C_1)^{-1}(\voltage C_2)  =  \gamma_2 , 
	 \quad
	 (\voltage C_0)^{-1}(\voltage C_2)  =  \gamma_1 \gamma_2 
	. \end{align*}
By the assumption of this \namecref{aG'=bG'-G'Case}, we know that $\gamma_1$ and~$\gamma_2$ generate~$G'$. Also, since $|G|$ is odd, we know that no element of~$G$ inverts any nontrivial element of~$G'$, so $\gamma_1 \gamma_2$ also generates~$G'$. Therefore, \fullcref{MarusicMethod34}{3} applies.

\begin{case} \label{aG'=bG'Pf-noteq}
Assume $\langle \gamma \rangle \neq G'$.
\end{case}
Since $S$ is minimal, we know $\langle \gamma \rangle$ contains either $\integer_{p^\mu}$ or $\integer_{q^\nu}$. By the assumption of this \namecref{aG'=bG'Pf-noteq}, we know it does not contain both. So let us assume $\langle \gamma \rangle = N \times \integer_{q^\nu}$, where $N$ is a proper subgroup of~$\integer_{p^\mu}$. 

Assume, for the moment, that $G/(G')^p$ is not the nonabelian group of order~$27$ and exponent~$3$. We use \cref{MarusicMethod}, so \cref{FreeLunch} allows us to assume $G' = \integer_p \times \integer_q$. Applying \cref{G'=p} to $G/\integer_q$ provides us with hamiltonian cycles $C_1$ and~$C_2$ in $\Cay \bigl( G/G' ; S \smallsetminus \{b\} \bigr)$, such that $\bigl\langle (\voltage C_1)^{-1} (\voltage C_2) \bigr\rangle$ contains $\integer_p$. (Furthermore, the two cycles have an oriented edge in common.) Since $S$ is a minimal generating set, we know that $C_i$ contains an edge labelled $a^{\pm1}$. (In fact, more than one, so we can take one that is not the edge in common with the other cycle.) Assume, without loss of generality, that it is labelled~$a$. Replacing this edge with~$b$ results in a hamiltonian cycle~$C_i'$, such that $\bigl\langle (\voltage C_i)^{-1}(\voltage C_i') \bigr\rangle = \langle \gamma \rangle = \integer_q$. Then \fullcref{MarusicMethod34}{4} applies.

We may now assume that $G/(G')^p$ is the nonabelian group of order~$27$ and exponent~$3$. Then $G/\langle \gamma \rangle$ is a $3$-group, so \cref{pgrp} tells us there is a directed hamiltonian cycle $C_0$ in the Cayley digraph $\overrightarrow{\Cay}\bigl( G/ \langle \gamma \rangle ; S \smallsetminus \{b\} \bigr)$. Since $S\smallsetminus \{b\} $ is a minimal generating set of $G/ \langle \gamma \rangle$, there must be at least two edges~$\alpha_1$ and~$\alpha_2$ that are labeled~$a$ in~$C$. Now the proof of \cref{aG'=bG'-G'Case} applies (but with $\langle \gamma \rangle$ in the place of~$G'$).
\end{proof}

\section{\texorpdfstring{Proof of \cref{MAINTHM}}{Proof of the main theorem}}

\begin{assump} \label{MainThmPfAssumps} 
We always assume:
\noprelistbreak
	\begin{enumerate}
	\item The generating set~$S$ is minimal.
	\item $S \cap G' = \emptyset$ \csee{GenInG'}.
	\item $p$ and~$q$ are distinct \csee{KeatingWitteThm}.
	\item \label{MainThmPfAssumps-NotNilp} 
	$G$ is not nilpotent \csee{GhaderpourMorrisNilpotent}.
	This implies $G/(G')^{pq}$ is not nilpotent \cite[Satz 3.5, p.~270]{Huppert}.
	\item  \label{MainThmPfAssumps-aG'=bG'}
	There do not exist $a,b \in S \cup S^{-1}$ with $a \neq b$ and $aG' = bG'$ \csee{aG'=bG'}.
	\item \label{MainThmPfAssumps-notnormal}
	There does not exist $s \in S$, such that $G' \subseteq \langle s \rangle$ \csee{GenNormal}.
	\end{enumerate}
\end{assump}

\begin{rem}
We consider several cases that are exhaustive\refnote{exhaustive}
 up to permutations of the variables $a$, $b$, and~$c$, and interchanging $p$ and~$q$. Here is an outline of the cases: 
	\begin{itemize} \itemindent = 0.25in 
	\item There exist $a,b \in S$, such that $\langle [a,b] \rangle = G'$.
	\noprelistbreak
		\begin{itemize} \addtolength{\itemindent}{0.5in} 
		\item[\pref{bina}] $\quot{b} \in \langle \quot{a} \rangle$.
		\item[\pref{bnotina}] $\quot{b} \notin \langle \quot{a} \rangle$ and $|\quot{a}| \ge 5$. 
		\item[\pref{ab=G'}] $|\quot{a}| = |\quot{b}| = 3$ and $\langle \quot{a} \rangle \neq \langle \quot{b} \rangle$.
		\end{itemize}
	\item There exist $a,b,c \in S$, such that $\integer_{p^\mu} \subseteq \langle [a,b] \rangle$ and $\integer_{q^\nu} \subseteq \langle [a,c] \rangle$.
	\noprelistbreak
		\begin{itemize}  \addtolength{\itemindent}{0.5in} 
		\item[\pref{bandcina}] $\quot{b}, \quot{c} \in \langle \quot{a} \rangle$.
		\item[\pref{a<ab<abc}] $\langle \quot{a} \rangle \subsetneq \langle \quot{a}, \quot{b} \rangle \subsetneq \langle \quot{a}, \quot{b}, \quot{c} \rangle$.
		\item[\pref{aCentG'}] $a$ centralizes~$G'/(G')^{pq}$.
		\item[\pref{bcNotIna}] $\quot{b}, \quot{c} \notin \langle \quot{a} \rangle$.
		\item[\pref{abcRemainder}] $\quot{c} \in \langle \quot{a} \rangle$ and $\quot{b} \notin \langle \quot{a} \rangle$.
		\end{itemize}
	\item There do not exist $a,b,c \in S$, such that $\langle [a, b], [a,c] \rangle = G'$.
		 \pref{G'NotIn<abc>} 
	\end{itemize}
\end{rem}

\begin{CASE} \label{bina}
Assume there exist $a,b \in S$, such that $\langle [a,b] \rangle = G'$ and $\quot{b} \in \langle \quot{a}\rangle$. 
\end{CASE}

\begin{proof}
We use \cref{MarusicMethod34}, so there is no harm in assuming $S = \{a,b\}$. Then $\langle \quot{a}\rangle = \langle \quot{a}, \quot{b} \rangle = \quot{G}$. Furthermore, \cref{FreeLunch} allows us to assume $G' = \integer_{pq}$. Let $n = | \quot{a}| = |\quot{G}|$, fix $k$ with $\quot{b} = \quot{a}^k$, and choose $\gamma \in G'$, such that $b = a^k \gamma$. Note that 
	\begin{itemize}
	\item $a^n = e$ (since \cref{aCents->proper} implies that $a$ cannot centralize a nontrivial subgroup of~$G'$),
	and
	\item $\langle \gamma \rangle = G'$ (since $\langle a \rangle \ltimes \langle \gamma \rangle = \langle a,b \rangle = G$). 
	\end{itemize}
We may assume $1 \le k < n/2$, by replacing $b$ with its inverse if necessary. We may also assume $n \ge 5$ (otherwise, we must have $k = 1$, contrary to \fullcref{MainThmPfAssumps}{aG'=bG'}). Therefore $n -k-2 > 0$.

We have the following three hamiltonian cycles in $\Cay(\quot{G}; a,b)$:
$$ C_1 = (a^n),
\qquad C_2 = (a^{n-k-1}, b, a^{-(k-1)}, b),
\qquad C_3 = (a^{n-k-2}, b, a^{-(k-1)}, b, a) .$$
Their voltages are
	\begin{align*}
	\voltage C_1 
		&= a^n = e 
		, \\
	\voltage C_2
		&= a^{n-k-1}  b a^{-(k-1)} b
		= a^{n-k-1}   (a^k \gamma) a^{-(k-1)}  (a^k \gamma)
		= a^n \cdot a^{-1} \gamma a \gamma
		= \gamma^a \gamma 
		, \\
	\voltage C_3
		&= a^{n-k-2} b a^{-(k-1)} b a
		= a^{-1}(a^{n-k-1} b a^{-(k-1)} b ) a
		= (\voltage C_2)^ a
		.
	\end{align*}
Since $|G|$ is odd, we know that $a$ does not invert $\integer_p$ or~$\integer_q$. Therefore $\voltage C_2$ generates~$G'$. Hence, the conjugate $\voltage C_3$ must also generate~$G'$. Furthermore, as was mentioned above, we know that $a$ does not centralize any nontrivial element of~$G'$, so $(\voltage C_2) (\voltage C_3)^{-1}$ also generates~$G'$. (Also note that all three hamiltonian cycles contain the oriented edge $(a)$.) Hence, \fullcref{MarusicMethod34}{3} applies. 
\end{proof}

\begin{CASE} \label{bnotina}
Assume there exist $a,b \in S$, such that $\langle [a,b] \rangle = G'$ and $\quot{b} \notin \langle \quot{a} \rangle$. Also assume $|\quot{a}| \ge 5$. 
\end{CASE}

\begin{proof}[Proof\/ {\upshape(cf.\ proof of {\cite[Case 4.3]{KeatingWitte}})}]
We use \cref{MarusicMethod34}, so there is no harm in assuming $S = \{a,b\}$. Furthermore, \cref{FreeLunch} allows us to assume $G' = \integer_{pq}$. Let $d = |\quot{G}/ \langle \quot{a} \rangle|$, so there is some $r$ with $\quot{b}^d \quot{a}^r = \quot{e}$ and $0 \le r < |\quot{a}|$. We may assume $r \le |\quot{a}| - 2$, by replacing $b$ with its inverse if necessary.

Applying \cref{KW43} to the hamiltonian cycle $(b^{-d})$ yields hamiltonian cycles $C_0$, $C_1$, and~$C_2$ (since $2 = 5 - 3 \le |\quot{a}| - 3$). Note that all of these contain the oriented edge~$\quot{b}(b^{-1})$. Furthermore, the voltage of~$C_k$ is
	$$\voltage C_k = \pi [a^{-k}, b] \, [a^{-k}, b]^{a^{-1}} ,$$
where $\pi = \voltage C_0$ is independent of~$k$. 

Since $[a^{-1}, b]$ generates~$G'$, and $a$ does not invert any nontrivial element of~$G'$ (recall that $|G|$ is odd), it is easy to see that $G'$ is generated by the difference of any two of 
	$$ \text{$e$, \ $[a^{-1}, b]$, \ and \ 
	$[a^{-2}, b] = [a^{-1},b] [a^{-1},b]^{a^{-1}}$. } $$
Using again the fact that $a$ does not invert any element of~$G'$, this implies that $G'$ is generated by the difference of any two of the three voltages,
so \fullcref{MarusicMethod34}{3} applies.
\end{proof}

\begin{CASE} \label{ab=G'}
Assume there exist $a,b \in S$, such that $\langle [a,b] \rangle = G'$, $|\quot{a}| = |\quot{b}| = 3$ and $\langle \quot{a} \rangle \neq \langle \quot{b} \rangle$.
\end{CASE}

\begin{proof}
This proof is rather lengthy. It can be found in \cref{a=b=3}.
\end{proof}

\begin{assump}
Henceforth, we assume there do not exist $a,b \in S \cup S^{-1}$, such that $\langle [a,b] \rangle = G'$.
\end{assump}

\begin{CASE} \label{bandcina}
Assume $\integer_{p^\mu} \subseteq \langle [a,b] \rangle$, 
$\integer_{q^\nu} \subseteq \langle [a,c] \rangle$, 
and
$\langle \quot{b}, \quot{c} \rangle \subseteq \langle \quot{a} \rangle$.
\end{CASE}

\begin{proof}
We use \cref{MarusicMethod34}, so there is no harm in assuming $S = \{a,b,c\}$. (Furthermore, \cref{FreeLunch} allows us to assume $G' = \integer_{pq}$, so $\langle [a,b] \rangle = \integer_p$ and $\langle [a,c] \rangle = \integer_q$.)
Then, since $\quot{b}, \quot{c} \in \langle \quot{a} \rangle$, we must have $\langle \quot{a} \rangle = \quot{G}$. Therefore, \cref{aCents->proper} tells us that $a$ does not centralize any nonidentity element of~$G'$.
Fix $k$ and~$\ell$ with $\quot{b} = \quot{a}^k$ and $\quot{c} = \quot{a}^\ell$. We may write $b = a^k \gamma_1$ and $c = a^\ell \gamma_2$, for some $\gamma_1 \in \integer_p$ and $\gamma_2 \in \gamma_q$.\refnote{b=agamma}

Since $1$, $k$, and~$\ell$ are distinct \fullcsee{MainThmPfAssumps}{aG'=bG'}, we may assume $1 < k < \ell < n/2$, by interchanging $b$ and~$c$ and/or replacing $b$ and/or~$c$ with its inverse if necessary. Therefore $\ell \ge 3$ and $k + \ell \le n - 2$, so we have the following three hamiltonian cycles in $\Cay( \quot{G}; a,b,c)$:
	\begin{align*}
	C_1 &= (a^{-n}) \\
	C_2 &= (a^{-(\ell-1)}, c, \ b, a^{-(k-1)}, b, \ a^{n-k-\ell - 2}, \ c ) \\
	C_3 &= (a^{-(\ell-2)}, c, \ b, a^{-(k-1)}, b, \ a^{n-k-\ell - 2}, \ c , a^{-1})
	. \end{align*}
Note that each of these contains the oriented edge $(a^{-1})$.

Since $a$ does not centralize any nonidentity element of~$G'$, we know $\voltage C_1 = e$. A straightforward calculation\refnote{Calculate:bandcina} shows
	$$ \voltage C_2 = (\gamma_1 \gamma_1^{a^{-1}})^{a^{-k-1}}   ( \gamma_2^{a^{-1}} \gamma_2) ,$$
which generates~$G'$.\refnote{GensG':bandcina} Therefore, 
	$\voltage C_3 = (\voltage C_2)^{a^{-1}}$ and $(\voltage C_2)^{-1} (\voltage C_3)$ also generate~$G'$. (For the latter, note that $a^{-1}$ does not centralize any nonidentity element of~$G'$.) Therefore \fullcref{MarusicMethod34}{3} applies.
\end{proof}

\begin{CASE} \label{a<ab<abc}
Assume $\integer_{p^\mu} \subseteq \langle [a,b] \rangle$, 
$\integer_{q^\nu} \subseteq \langle [a,c] \rangle$, 
and
there exists $s \in \{a,b\}$, such that 
$\langle \quot{a} \rangle \subsetneq \langle \quot{a}, \quot{s} \rangle \subsetneq  \langle \quot{a}, \quot{b} , \quot{c} \rangle$.
\end{CASE}

\begin{proof} 
We use \cref{MarusicMethod34}, so there is no harm in assuming $S = \{a,b,c\}$. Furthermore, \cref{FreeLunch} allows us to assume $G' = \integer_{pq}$, so $\langle [a,b] \rangle = \integer_p$ and $\langle [a,c] \rangle = \integer_q$. 
Choose $A,B,C \ge 3$, such that $\quot{a}^A = \quot{e}$, and every element of~$\quot{G}$ can be written uniquely in the form 
	$$ \text{$\quot{a}^x \quot{b}^y \quot{c}^z$ \ with \ $ \begin{matrix}
	\text{$0 \le x < A$}, \\ \text{$0 \le y < B$}, \\ \text{$0 \le z < C$}
	.
	\end{matrix}$} $$
More precisely, we may let 
	$$\begin{cases}
	\text{
	$A = |\quot{a}|$, 
	$B = |\langle \quot{a}, \quot{b} \rangle: \langle \quot{a} \rangle|$,  
	$C = |\quot{G} : \langle \quot{a}, \quot{b} \rangle|$
	}
	& \text{if $s = b$} ,\\
	\text{
	$A = |\quot{a}|$, 
	$C = |\langle \quot{a}, \quot{c} \rangle: \langle \quot{a} \rangle|$,  
	$B = |\quot{G} : \langle \quot{a}, \quot{c} \rangle|$
	}
	& \text{if $s = c$} 
	. \end{cases}$$
Then we have the following hamiltonian cycle~$X$ in $\Cay(\quot{G}; a,b,c)$ \csee{3DFig}:
	\begin{align*}
	 X = \biggl( a, 
	 	&
		\Bigl( a^{A-2}, (b, a^{-(A-1)}, b, a^{A-1})^{(B-1)/2} , c, 
		\\[-10pt] &\hskip 0.5in
		(a^{-(A-1)}, b^{-1}, a^{A-1}, b^{-1})^{(B-1)/2}, a^{-(A-2)}, c \Bigr)^{(C-1)/2},
		\\&b, a^{-1}, b^{B-2}, a,
		(a^{A-2}, b^{-1}, a^{-(A-2)}, b^{-1})^{(B-3)/2},
		\\ & a^{A-2}, b^{-1}, a^{-(A-3)}, b^{-1}, a^{A-2}, c^{-(C-1)} \biggr)
	. \end{align*} 
\begin{figure}[t]
\begin{center}
\includegraphics{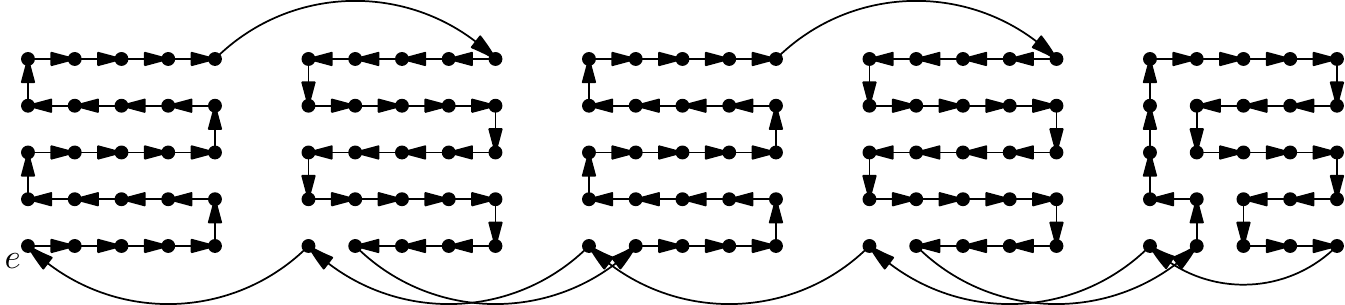}
\caption{A hamiltonian cycle~$X$.}
\label{3DFig}
\end{center}
\end{figure}

We obtain a new hamiltonian cycle $X^p$ by replacing a subpath of the form $[g]\bigl(a^{A-1}, b, a^{-(A-1)}\bigr)$ with $[g]\bigl(a^{-(A-1)}, b, a^{A-1}\bigr)$. Then $(\voltage X)^{-1} (\voltage X^p)$ is a conjugate of\refnotelower{voltXXp}
	\begin{align*}
	\bigl( a^{A-1} b a^{-(A-1)} \bigr)^{-1} \bigl( a^{-(A-1)} b a^{A-1} \bigr) 
	= [b,a^{A-1}]^{a} [b,a^{A-1}] 
	. \end{align*}
Similarly, replacing a subpath of the form $[g] \bigl( a^{A-1}, c, a^{-(A-1)} \bigr)$ with $[g] \bigl( a^{-(A-1)}, c, a^{A-1} \bigr)$ results in a hamiltonian cycle~$X_q$, such that $(\voltage X)^{-1} (\voltage X_q)$ is a conjugate of $[c,a^{A-1}]^a [c,a^{A-1}] $. Furthermore, doing both replacements results in a hamiltonian cycle~$X^p_q$, such that $(\voltage X^p)^{-1} (\voltage X^p_q)$ is also a conjugate of $[c,a^{A-1}]^a [c,a^{A-1}] $. Note that all four of these hamiltonian cycles contain the oriented edge $c(c^{-1})$.

Since $G' \not\subseteq \langle a \rangle$ \fullcsee{MainThmPfAssumps}{notnormal}, we may assume $a^A \in \integer_p$ (by interchanging $p$ and~$q$ if necessary). Since $[c,a] \in \integer_q$, this implies that $c$ centralizes~$a^A$, so $[c,a^{A-1}] = [c,a^{-1}]$ generates~$\integer_q$. Since $a$ does not invert any nontrivial element of~$\integer$ (recall that $G$ has odd order), this implies that $[c,a^{A-1}]^a [c,a^{A-1}]$ generates~$\integer_q$. 

Assume, for the moment, that $[b,a^{A-1}]$ generates~$\integer_p$. Since $a$ does not invert any nontrivial element of~$\integer_p$, this implies that $[b,a^{A-1}]^a [b,a^{A-1}]$ generates~$\integer_p$. 
Therefore, \fullcref{MarusicMethod34}{4} applies.

We may now assume $[b,a^{A-1}]$ does not generate~$\integer_p$. This means $[b,a^{A-1}] = e$. Since $[b, a^{-1}] \neq e$, we conclude that $[b, a^A] \neq e$, so 
	$$ \text{$b$ does not centralize~$\integer_p$.} $$
We have the following hamiltonian cycle~$Y_1$ in $\Cay(\quot{G}; a,b,c)$ \csee{3DFigY}:
	\begin{align*}
	 Y_1 = \biggl( b, &
	 	\Bigl( b^{B-3}, (a, b^{-(B-2)}, a, b^{B-2})^{(A-1)/2}, b, a^{-(A-1)}, c,
	 	\\ & \hskip 0.5in a^{A-1}, b^{-1}, (b^{-(B-2)}, a^{-1}, b^{B-2}, a^{-1} )^{(A-1)/2}, b^{-(B-3)}, c \Bigr)^{(C-1)/2},
		\\ & b^{B-2}, a, (a^{A-2}, b^{-1}, a^{-(A-2)}, b^{-1})^{(B-1)/2}, a^{A-1}, c^{-(C-1)} \biggr)
	. \end{align*} 
\begin{figure}[t]
\begin{center}
\includegraphics[scale=0.95]{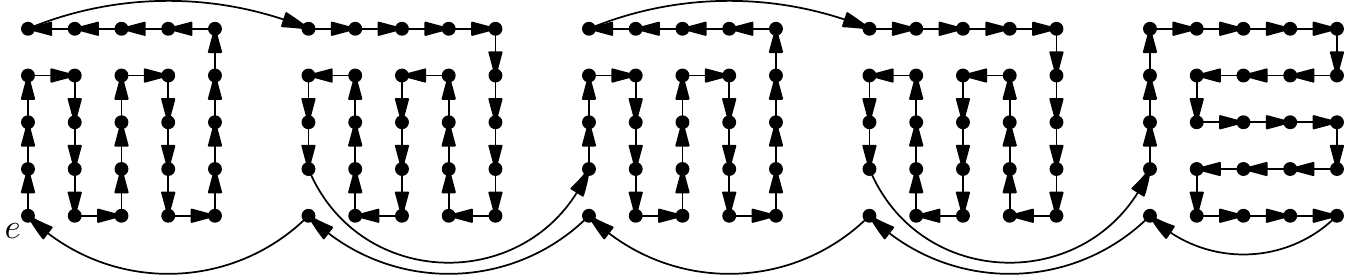}
\caption{A hamiltonian cycle~$Y_1$.}
\label{3DFigY}
\end{center}
\end{figure}

We create a new hamiltonian cycle~$Y_2$ by replacing a subpath of the form $[g]\bigl( a^{-(A-1)}, c, a^{A-1} \bigr)$ with $[g]\bigl( a^{A-1}, c, a^{-(A-1)} \bigr)$. This is the same as the construction of~$X_q$ from~$X$, but with $a$ and~$a^{-1}$ interchanged, so the same calculation shows $(\voltage Y_1)^{-1} (\voltage Y_2)$ is a conjugate of $[c, a^{-(A-1)}]^{a^{-1}} [c, a^{-(A-1)}]$, which generates~$\integer_q$. Furthermore, since $Y_1$ and~$Y_2$ both contain the oriented path $[b^{B-3}](b, a, b^{-1})$, and either the oriented edge $[b^{B-2}](a)$ or the oriented edge $[b^{B-2}a](a^{-1})$, \cref{StandardAlterationIsG'} provides hamiltonian cycles $Y_1'$ and~$Y_2'$, such that $(\voltage Y_i)^{-1}(\voltage Y_i')$ generates~$\integer_p$. Since all four hamiltonian cycles contain the oriented edge $[c](c^{-1})$, \fullcref{MarusicMethod34}{4} applies.
\end{proof}

\begin{CASE} \label{aCentG'}
Assume $\integer_{p^\mu} \subseteq \langle [a,b] \rangle$, 
$\integer_{q^\nu} \subseteq \langle [a,c] \rangle$, 
and
$a$~centralizes $G'/(G')^{pq}$.
\end{CASE}

\begin{proof}
We use \cref{MarusicMethod34}, so there is no harm in assuming $S = \{a,b,c\}$. Furthermore, \cref{FreeLunch} allows us to assume $G' = \integer_{pq}$, so $\langle [a,b] \rangle = \integer_p$ and $\langle [a,c] \rangle = \integer_q$. 

Note that $[a,b^{-1},c] \in \integer_p$, $[c,a^{-1},b] \in \integer_q$, and $[b,c^{-1},a] = e$ (because $a$ centralizes~$G'$). Since $\integer_p \cap \integer_q = \{e\}$, and the Three-Subgroup Lemma \cite[Thm.~2.3, p.~19]{Gorenstein-FinGrps} tells us 
	$$ [a,b^{-1},c]^b [b, c^{-1}, a]^c [c, a^{-1},b]^a = e ,$$
we conclude that $[a,b^{-1},c] = [c,a^{-1},b] = e$,\refnote{3SubgrpLem}
 so
	$$ \text{$c$ centralizes~$\integer_p$ and $b$ centralizes~$\integer_q$} .$$

We know $G' \not\subseteq Z(G)$, because $G$ is not nilpotent \fullcsee{MainThmPfAssumps}{NotNilp}.
Since $a$~centralizes~$G'$, this implies we may assume $c$ does not centralize~$G'$ (by interchanging $b$ and~$c$ if necessary). So $c$ does not centralize~$\integer_q$. Since $a$, $b$, and~$G'$ all centralize~$\integer_q$, this implies $c \notin \langle a,b,G' \rangle$. In other words, $\quot{c} \notin \langle \quot{a}, \quot{b} \rangle$.
Furthermore, applying \cref{aCents->proper} to the group $\langle a, b \rangle$ tells us that $\langle \quot{a} \rangle \neq \langle \quot{a}, \quot{b} \rangle$.
Therefore $\langle \quot{a} \rangle \subsetneq \langle \quot{a}, \quot{b} \rangle \subsetneq \langle \quot{a}, \quot{b}, \quot{c} \rangle$, so \cref{a<ab<abc} applies.
\end{proof}

\begin{CASE} \label{bcNotIna}
Assume $\integer_{p^\mu} \subseteq \langle [a,b] \rangle$, 
$\integer_{q^\nu} \subseteq \langle [a,c] \rangle$, 
and
$\quot{b},\quot{c} \notin \langle \quot{a} \rangle$.
\end{CASE}

\begin{proof}
We use \cref{MarusicMethod34}, so there is no harm in assuming $S = \{a,b,c\}$. Furthermore, \cref{FreeLunch} allows us to assume $G' = \integer_{pq}$, so $\langle [a,b] \rangle = \integer_p$ and $\langle [a,c] \rangle = \integer_q$. 
We may assume $\langle \quot{a}, \quot{b} \rangle = \langle \quot{a}, \quot{c} \rangle = \quot{G}$, for otherwise \cref{a<ab<abc} applies.

Let us begin by showing that $a$ does not centralize any nontrivial element of~$G'$. Suppose not. Then we may assume that $a$ centralizes~$\integer_p$. Let $\underline{G} = G/\integer_q = G/ \langle [a,c] \rangle$. Since $\langle a, c, G' \rangle = G$, we know that $\langle \underline{a}, \underline{c}, \underline{\integer_p} \rangle = \underline{G}$ , so $\underline{a}$ is in the center of~$\underline{G}$. This contradicts the fact that $\langle [\underline{a}, \underline{b}] \rangle = \underline{\integer_p}$ is nontrivial.

Since $\quot{G}$ is abelian (and because $\quot{b},\quot{c} \notin \langle \quot{a} \rangle$), it is easy to choose a hamiltonian cycle $(s_i)_{i=1}^d$ in $\Cay(\quot{G}/\langle \quot{a} \rangle; S)$ that contains both an edge labeled~$b$ (or~$b^{-1}$) and an edge labeled~$c$ (or~$c^{-1}$). Note that
	$$ C_0 = \bigl( (s_i)_{i=1}^{d-1}, a^{|\quot{a}|-1},
		(s_{d - 2i+1}^{-1}, a^{-(|\quot{a}|-2)}, s_{d - 2i}^{-1}, a^{|\quot{a}|-2})_{i=1}^{(d-1)/2},
		a \bigr) $$
is a hamiltonian cycle in $\Cay( \quot{G} ; S)$. 
		
\begin{subcase}
Assume $|\quot{a}| > 3$. 
\end{subcase}
We may assume $s_1 = b^{-1}$ and $s_2 = c^{-1}$. Then $C_0$ contains the four subpaths
	$$ \text{$(b^{-1})$, 
	\ $[b^{-1}a^2](a^{-1}, b, a)$,  \quad 
	$[b^{-1}](c^{-1})$,
	\  $[b^{-1}c^{-1}a^{-2}](a, c, a^{-1})$}
	. $$
Therefore, we may let $g$ be either $b^{-1}$ or $b^{-1}c^{-1}$ in \cref{StandardAlteration}, so \fullcref{StandardAlterationIsG'}{NoCent} tells us we have hamiltonian cycles $C^b$ and $C^c$, such that $(\voltage C_0)^{-1} (\voltage C^b)$ is a generator of~$\integer_p$, and $(\voltage C_0)^{-1} (\voltage C^c)$ is a generator of~$\integer_q$.
Since $|\quot{a}| > 3$, we see that $C^b$, like~$C_0$, contains 
$[b^{-1}](c^{-1})$ and $[b^{-1}c^{-1}a^{-2}](a, c, a^{-1})$, so \fullcref{StandardAlterationIsG'}{NoCent} provides a hamiltonian cycle $C^b_c$, such that $(\voltage C^b)^{-1} (\voltage C^b_c)$ is a generator of~$\integer_q$. Therefore, \fullcref{MarusicMethod34}{4} applies (since each of these four hamiltonian cycles contains the oriented edge $[a^{-1}](a)$).

\begin{subcase}
 Assume $d > 3$. 
\end{subcase}
We may assume $s_1 = b^{-1}$ and $s_3 = c^{-1}$. Then $C_0$ contains the four subpaths
	$$(b^{-1}), \ [b^{-1}a^2](a^{-1}, b, a), \quad
	[s_1s_2](c^{-1}), \ [s_1s_2c^{-1}a^2] (a^{-1}, c, a)
	. $$
Therefore, we may let $g$ be either $b^{-1}$ or $s_1s_2c^{-1}$ in \cref{StandardAlteration}, so \fullcref{StandardAlterationIsG'}{NoCent} tells us we have hamiltonian cycles $C^b$ and $ C^c$, such that $(\voltage C_0)^{-1} (\voltage C^b)$ is a generator of~$\integer_p$, and $(\voltage C_0)^{-1} (\voltage C^c)$ is a generator of~$\integer_q$.
It is clear that $C^b$, like~$C_0$, contains $[s_1s_2](c^{-1})$ and $[s_1s_2c^{-1}a^2](a^{-1}, c, a)$, so \fullcref{StandardAlterationIsG'}{NoCent}  provides a hamiltonian cycle $C^b_c$, such that $(\voltage C^b)^{-1} (\voltage C^b_c)$ is a generator of~$\integer_q$. Therefore, \fullcref{MarusicMethod34}{4} applies (since each of these four hamiltonian cycles contains the oriented edge $[a^{-1}](a)$).

\begin{subcase}
 Assume $|\quot{a}| = 3$ and $d = 3$. 
 \end{subcase}
Since $d = 3$, we may assume $\quot{b} \equiv \quot{c} \pmod{\langle \quot{a} \rangle}$ (by replacing $c$ with its inverse if necessary). 
Let 
	$$C_0 = (b^{-1},c^{-1},a^2, c, a^{-1}, b, a^2) ,$$
so $C_0$ is a hamiltonian cycle in $\Cay(\quot{G}; S)$. 
Then $C_0$ contains the four subpaths 
	$$ (b^{-1}), \ [b^{-1}a^2](a^{-1}, b, a) , \quad 
	[b^{-1}](c^{-1}), \ [b^{-1} c^{-1}a^{-2}](a, c, a^{-1}) .$$
Therefore, we may let $g$ be either $b^{-1}$ or $b^{-1}c^{-1}$ in \cref{StandardAlteration}, so \fullcref{StandardAlterationIsG'}{NoCent} tells us we have hamiltonian cycles 
	$$C^b = (a,b^{-1},a^{-1},c^{-1},a^2, c, b, a)$$
and	
	$$C^c = (b^{-1},a^{-1}, c^{-1},a^2, c, b, a^2),$$
such that $(\voltage C_0)^{-1} (\voltage C^b)$ is a generator of~$\integer_p$, and $(\voltage C_0)^{-1} (\voltage C^c)$ is a generator of~$\integer_q$.
Furthermore, $C^c$ contains the oriented paths $[ab^{-1}](b)$ and $[a^{-1}](a, b^{-1}, a^{-1})$, so, by letting $g = a$ in \cref{StandardAlteration} (and replacing $b$ with~$b^{-1}$), \fullcref{StandardAlterationIsG'}{NoCent} tells us we have a hamiltonian cycle 
	$$C^c_b = (a^2, b^{-1}, c^{-1}, a^2, c, a^{-1}, b ),$$
such that $(\voltage C^c)^{-1} (\voltage C^c_b)$ is a generator of~$\integer_p$. Therefore \fullcref{MarusicMethod34}{4} applies (since all four of these hamiltonian cycles contain the oriented edge $[b^{-1} c^{-1}](a)$).
\end{proof}

\begin{CASE} \label{abcRemainder}
Assume $\integer_{p^\mu} \subseteq \langle [a,b] \rangle$, 
$\integer_{q^\nu} \subseteq \langle [a,c] \rangle$,
 $\quot{c} \in \langle \quot{a} \rangle$,
and $\quot{b} \notin \langle \quot{a} \rangle$.
\end{CASE}

\begin{proof}
We use \cref{MarusicMethod}, so there is no harm in assuming $S = \{a,b,c\}$. Furthermore, \cref{FreeLunch} allows us to assume $G' = \integer_{pq}$, so $\langle [a,b] \rangle = \integer_p$ and $\langle [a,c] \rangle = \integer_q$. 
Also note that, from \fullcref{MainThmPfAssumps}{aG'=bG'}, we know $\quot{c} \notin \{\quot{a}^{\pm1}\}$, so we must have $|\quot{a}| > 3$. 

Let $d = |\quot{G} / \langle \quot{a} \rangle|$.
Since $\quot{c} \in \langle \quot{a} \rangle$, we have $\langle \quot{a}, \quot{b} \rangle = \quot{G}$, so $(b^{d})$ is a hamiltonian cycle in $\Cay( \quot{G} / \langle \quot{a} \rangle ; S )$. 
Choose $r$ such that $a^r b^{d} \in G'$ and $0 \le r \le |\quot{a}|-1$. Assume $r < |\quot{a}|/2$ (so $r \le |\quot{a}| - 3$), by replacing $b$ with its inverse if necessary. Then letting $k =  |\quot{a}|-3$ in \cref{KW43} provides us with a hamiltonian cycle $C^0 = C_{|\quot{a}|-3}$. 

Choose $\ell$ with $\quot{c} = \quot{a}^\ell$, and write $c = a^\ell \gamma$, where $\integer_q \subseteq \langle \gamma \rangle$. We may assume $0 \le \ell < |\quot{a}|/2$ (by replacing $c$ with its inverse, if necessary).
Then $\ell \le |\quot{a}| - 3$, so we see from \cref{KW43Fig} that $C_{|\quot{a}|-3}$ contains the path $[a^\ell b](a^{-(\ell+1)})$. Replacing this with the path $[a^\ell b](c^{-1}, a^{\ell-1}, c^{-1})$ results in a hamiltonian cycle $C^1$, such that $(\voltage C^0)^{-1}(\voltage C^1)$ is a conjugate of
	\begin{align*}
	  c^{-1} a^{\ell-1} c^{-1} \cdot a^{\ell+1} 
	&= (a^\ell \gamma)^{-1} a^{\ell-1} (a^\ell \gamma)^{-1} \cdot a^{\ell+1}
	=  \gamma^{-1} ( \gamma^{-1})^a
	. \end{align*}
Since $|G|$ is odd, we know that $a$ does not invert any nontrivial element of~$G'$, so this is a generator of~$\langle \gamma \rangle$, which contains $\langle [a,c] \rangle = \integer_q$.

Furthermore, from \cref{KW43Fig}, we see that $C_{|\quot{a}|-3}$ contains both the oriented edge $[b^{-1}a^{-1}](b)$ and the oriented path $[b^{-1} a](a^{-1}, b, a)$. Then, by construction, $C^1$ also contains these paths. Therefore, we may apply \cref{StandardAlteration} with $g = b^{-1}a^{-1}$, so \fullcref{StandardAlterationIsG'}{NoInvert} tells us we have hamiltonian cycles $\widehat C^0$ and $\widehat C^1$, such that $(\voltage C^i)^{-1} (\voltage \widehat C^i)$ is a generator of~$\integer_p$. Therefore \fullcref{MarusicMethod34}{4} applies (since there are many oriented edges, such as $[a^{-1}](a^{-1})$, that are in all four hamiltonian cycles).
\end{proof}

\begin{CASE} \label{G'NotIn<abc>}
Assume there do not exist $a,b,c \in S$, such that $\langle [a,b], [a,c] \rangle = G'$.
\end{CASE}

\begin{proof}
Let $\underline{G} = G/(G')^{pq}$, so $\underline{G}' = \integer_{pq}$.
The assumption of this \namecref{G'NotIn<abc>} implies that we may partition $S$ into two nonempty sets $S_p$ and~$S_q$, such that\refnotelower{SpSq}
	\noprelistbreak
	\begin{itemize}
	\item  $\underline{S_p}$ centralizes~$\underline{S_q}$ in~$\underline{G}$,
	and
	\item for $r \in \{p,q\}$,  and $a,b \in S_r$, we have $[\underline{a},\underline{b}] \in \underline{\integer_r}$.
	\end{itemize}
Let $G_p = \langle S_p \rangle$, $G_q = \langle S_q \rangle$, and $Z = \underline{G_p} \cap \underline{G_q} \subseteq Z(\underline{G})$.

Since $\underline{G}$ is not nilpotent \fullcsee{MainThmPfAssumps}{NotNilp}, we know that $\underline{G}' \not\subseteq Z(\underline{G})$. Therefore, we may assume $\integer_q \not\subseteq Z(\underline{G})$ (by interchanging $p$ and~$q$ if necessary). Since $\underline{G_p} \cap \underline{G_q} \subseteq Z(\underline{G})$, this implies $\integer_q \not\subseteq \underline{G_p}$.

\begin{subcase} \label{bbMinModaa}
Assume there exist $a_p,b_p, a_q,b_q \in S$, such that
	 $\langle [\underline{a_p}, \underline{b_p}] \rangle = \underline{\integer_p}$,
	 $\langle [\underline{a_q}, \underline{b_q}] \rangle = \underline{\integer_q}$,
	and 
	 $\{ b_p, b_q \}$ is a minimal generating set of $\langle \quot{a_p}, \quot{b_p}, \quot{a_q}, \quot{b_q} \rangle / \langle \quot{a_p}, \quot{a_q} \rangle$.
\end{subcase}
We use \cref{MarusicMethod} with $S_0 = \{a_p,b_p, a_q,b_q\}$. Assume, for simplicity, that $S = S_0$. \Cref{FreeLunch} allows us to assume $G' = \integer_{pq}$, so $G = \underline{G}$.

After perhaps replacing some generators with their inverses, it is easy to find:
	\begin{itemize}
	\item a hamiltonian cycle $(s_i)_{i=1}^m$ in $\Cay \bigl(  \langle \quot{a_p}, \quot{a_q} \rangle ; a_p, a_q \bigr)$, such that $s_{m-2} = a_p$ and $s_{m-1} = a_q$,
	and
	\item a hamiltonian cycle $(t_j)_{j=1}^n$ in $\Cay \bigl( \quot{G} / \langle \quot{a_p}, \quot{a_q} \rangle ; b_p, b_q \bigr)$, such that $t_1 = b_p$ and $t_3 = b_q$.
	\end{itemize}
We have the following hamiltonian cycle~$C_0$ in $\Cay(G;S)$:\refnotelower{Gaa}
	$$ C_0 = \Bigl( 
		\bigl( (s_i)_{i=1}^{n-2}, t_{2j-1}, (s_{n-1-i}^{-1})_{i-1}^{n-2}, t_{2j} \bigr)_{j=1}^{(m-1)/2},
		(s_i)_{i=1}^{n-1},  (t_{m-j}^{-1})_{j=1}^{m-1}, s_n \Bigr) .$$
Much as in the proof of \cref{StandardAlteration}, we construct a hamiltonian cycle~$C_1$ by
	\begin{itemize}
	\item replacing the oriented edge $[s_m^{-1}b_p](b_p^{-1})$ with the path $[s_m^{-1}b_p](a_q^{-1}, b_p^{-1}, a_q)$,
	and
	\item the oriented path $[s_m^{-1}a_q^{-1}a_p^{-1}](a_p, b_p, a_p^{-1})$ with $[s_m^{-1}a_q^{-1}a_p^{-1}](b_p)$. 
	\end{itemize}
Then there exist $g,h \in G$, such that\refnotelower{GaaVoltage}
	$$(\voltage C_0)^{-1}(\voltage C_1)
	= [b_p^{-1}, a_q]^g \, [a_p^{-1}, b_p]^{h} 
	=   e^g \cdot [a_p^{-1}, b_p]^{h}
	=  [a_p^{-1}, b_p]^{h}
	,$$
which generates~$\integer_p$.

Similarly, we may construct hamiltonian cycles $C_0'$ and~$C_1'$ from $C_0$ and~$C_1$ by
	\begin{itemize}
	\item replacing the oriented edge $[s_m^{-1}t_1t_2 b_q](b_q^{-1})$ with the path $[s_m^{-1}t_1t_2 b_q](a_q^{-1}, b_q^{-1}, a_q)$,
	and
	\item the oriented path $[s_m^{-1}a_q^{-1}a_p^{-1}t_1t_2](a_p, b_q, a_p^{-1})$ with $[s_m^{-1}a_q^{-1}a_p^{-1}t_1t_2](b_q)$. 
	\end{itemize}
Then, for $k \in \{0,1\}$, essentially the same calculation shows there exist $g',h' \in G$, such that 
	$$(\voltage C_k)^{-1}(\voltage C_k')
	= [b_q^{-1}, a_q]^{g'} [a_p^{-1}, b_q]^{h'} 
	= [b_q^{-1}, a_q]^{g'} \cdot e^{h'} 
	= [b_q^{-1}, a_q]^{g'}
	 ,$$
which generates~$\integer_q$.

All four hamiltonian cycles contain the oriented edge $(s_1)$, so \fullcref{MarusicMethod34}{4} applies.

\begin{subcase} \label{Not27}
Assume $\underline{G_p}$ is not the nonabelian group of order~$27$ and exponent~$3$. 
\end{subcase}
We will apply \cref{MarusicMethod34}, so \cref{FreeLunch} allows us to assume $G' = \integer_{pq}$, which means $\underline{G} = G$. 

\begin{claim}
We  may assume $S_q \cap (G' Z) = \emptyset$.
\end{claim}
Suppose $a_q \in S_q \cap (G' Z)$. By the minimality of~$S$, we know $a_q \notin G_p$. Since $Z$ and~$\integer_p$ are contained in~$G_p$, this implies $G' \subseteq \langle G_p , a_q \rangle$. Therefore, the minimality of~$S$ implies that $S_q \smallsetminus \{a_q\}$ is a minimal generating set of $\quot{G}/\langle \quot{G_p}, \quot{a_q} \rangle$. So \cref{bbMinModaa} applies. This completes the proof of the claim.

Now, applying \cref{G'=p/Z} to~$G_q$ tells us there exist hamiltonian cycles $C_q$ and~$C_q'$ in $\Cay \bigl( \quot{G_q}/\quot{Z} ; S_q \bigr)$, such that $C_q$ and~$C_q'$ have an oriented edge in common, and $\langle (\voltage C_q)^{-1}(\voltage C_q') \rangle = \integer_q$. 

Also, \cref{G'=p} provides hamiltonian cycles $C_p$ and~$C_p'$ in $\Cay \bigl( \quot{G_p} ; S_p \bigr)$, such that $C_p$ and~$C_p'$ have an oriented edge in common, and $\langle (\voltage C_p)^{-1}(\voltage C_p') \rangle = \integer_p$. 

For $r \in \{p,q\}$, write $C_r = (s_{r,i})_{i=1}^{n_r}$ and $C_r' = (t_{r,i})_{i=1}^{n_r}$. Since $C_r$ and~$C_r'$ have an edge in common, we may assume $s_{r,n_r} = t_{r,n_r}$. 

Let 
	\begin{align} \label{CForNot27}
	 C = \Bigl(
	(s_{p,i})_{i=1}^{n_p - 1}, (s_{q,i})_{i=1}^{n_q - 1}, 
	\bigl( 
		 s_{p, n_p - 2i +1}^{-1} , 
		 (s_{q,n_q-j}^{-1})_{j=1}^{n_q - 2},
		s_{p, n_p - 2i}^{-1} ,
		 (s_{q,j})_{j=2}^{n_q - 1}
	\bigr)_{i=1}^{(n_p-1)/2},
	s_{q,n_q}
	\Bigr)
	. \end{align}
Then $C$ is a hamiltonian cycle in $\Cay (\quot{G} ; S)$.

For $r \in \{p,q\}$, a path of the form $[g](s_{r,i})_{i=1}^{n_r - 1}$ appears near the start of~$C$. We obtain a new hamiltonian cycle~$C^r$ in $\Cay\bigl( \quot{G} ; S \bigr)$ by replacing this with $[g](t_{r,i})_{i=1}^{n_r - 1}$. We can also construct a hamiltonian cycle $C^{p,q}$ by making both replacements. Then
	$$ \text{$\langle ( \voltage C)^{-1} ( \voltage C^r) \rangle
	= \langle (\voltage C_r)^{-1}(\voltage C_r') \rangle = \integer_r$} ,$$
and
	$$ \langle ( \voltage C^q)^{-1} ( \voltage C^{p,q}) \rangle
	= \langle ( \voltage C_p)^{-1} ( \voltage C_p') \rangle
	= \integer_p ,$$
so \fullcref{MarusicMethod34}{4} applies (since all four hamiltonian cycles contain the oriented edge $[s_{q,n_q}^{-1}](s_{q,n_q})$).

\begin{subcase}
Assume $\underline{G_p}$ is the nonabelian group of order~$27$ and exponent~$3$. 
\end{subcase}
We have $p = 3$, and \fullcref{27grp}{mu=1} tells us $\mu = 1$; i.e., $G' = \integer_3 \times \integer_{q^\nu}$. Therefore $\underline{G} = G/(G')^q$.

Let $C_p = (s_{p,i})_{i=1}^{27}$ be a hamiltonian cycle in $\Cay \bigl( \underline{G_p}; S_p \bigr)$.
Also, for $r = q$, \cref{G'=p} provides hamiltonian cycles $C_q = (s_{q,i})_{i=1}^{n_q}$ and $C_q' = (t_{q,i})_{i=1}^{n_q}$ in $\Cay(\quot{G_q}; S_q)$, such that $s_{q,n_q} = t_{q,n_q}$ and $(\voltage C_q)^{-1} (\voltage C_q')$ generates~$\integer_{q^\nu}$. Define the hamiltonian cycle $C$ as in \pref{CForNot27} (with $n_p = 27$). We obtain a new hamiltonian cycle~$C^q$ in $\Cay\bigl( \quot{G} ; S \bigr)$ by replacing an occurrence of $(s_{q,i})_{i=1}^{n_q - 1}$ with the path $(t_{q,i})_{i=1}^{n_q - 1}$. Much as in \cref{Not27}, we have 
	$$ \langle (\underline{ \voltage C})^{-1} (\underline{ \voltage C^q}) \rangle 
	= \langle (\underline{\voltage C_q})^{-1}(\underline{\voltage C_q'}) \rangle
	= \underline{\integer_q} ,$$
so $\underline{\voltage C}$ and~$\underline{ \voltage C^q}$ cannot both be trivial. Therefore, applying \cref{FGL} with $N = \integer_q$ provides a hamiltonian cycle in $\Cay(\underline{G}; S)$, and then \cref{FreeLunch} tells us there is a hamiltonian cycle in $\Cay(G;S)$. 
\end{proof}

\section{\texorpdfstring{Proof of \cref{ab=G'}}{Proof of the omitted case}} \label{a=b=3}

In this section, we prove \cref{ab=G'}. Therefore, the following assumption is always in effect:

\begin{assump} \label{a=b=3Assump}
Assume there exist $a,b \in S$, such that $\langle [a,b] \rangle = G'$, $|\quot{a}| = |\quot{b}| = 3$, and $\langle \quot{a} \rangle \neq \langle \quot{b} \rangle$.
\end{assump}

We consider two cases:

\begin{RomanCase}
Assume $\#S > 2$.
\end{RomanCase}

\begin{proof}
Let $c$ be a third element of~$S$, and let $\ell = |\quot{G} : \langle \quot{a}, \quot{b} \rangle|$. (Since $S$ is a minimal generating set, and $G' = \langle [a,b] \rangle \subseteq \langle a, b \rangle$, we must have $\ell > 1$.) We use \cref{MarusicMethod} with $S_0 = \{a,b,c\}$; assume, for simplicity, that $S = S_0$. \Cref{FreeLunch} allows us to assume $G' = \integer_{pq}$.
Let 
	$$ (s_i)_{i=1}^{3\ell} = \bigl( (b, c, b^{-1}, c)^{(\ell-1)/2}, b^2, c^{-(\ell-1)}, b \bigr) , $$
so $(s_i)_{i=1}^{3\ell}$ is a hamiltonian cycle in $\Cay \bigl( \quot{G}/\langle \quot a \rangle ; b, c  \bigr)$. Note that 
	$$s_1 = s_5 = b .$$
From the definition of $(s_i)_{i=1}^{3\ell}$, it is easy to see that
	$ \quot{ \prod_{i=1}^{3\ell} s_i } = \quot{b}^3 = \quot{e} $,
so we have the following hamiltonian cycle~$C_0$ in $\Cay ( \quot{G} ; a, b, c  )$ \csee{HardabcFig}:
	\begin{align*}
	 C_0
	=  \bigl( 
	&(s_j)_{j=1}^{3\ell-3}, a^{-1}, s_{3\ell-2}, s_{3\ell-1},  a^{-1}, s_{3\ell}, 
	 \\& ( a, s_{2j-1}, a^{-1}, s_{2j})_{j=1}^{3(\ell-1)/2},
	s_{3\ell-2},
	 a^{-1}, s_{3\ell-1}, s_{3\ell}
	 \bigr) 
	. \end{align*}
\begin{figure}[t]
\begin{center}
\includegraphics{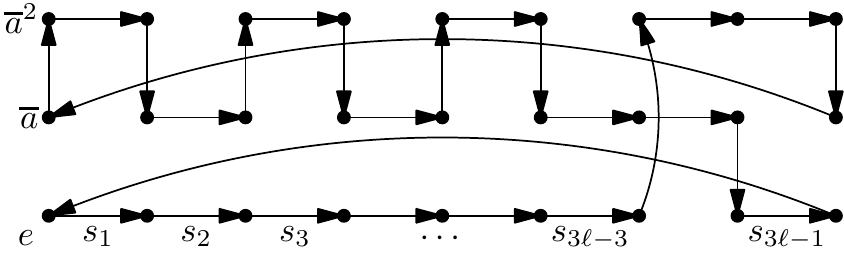}
\caption{A hamiltonian cycle~$C_0$.}
\label{HardabcFig}
\end{center}
\end{figure}
%
%
%
%
%
%
%
%
%
%
%
%
%
%
%

Since $s_1= b$, we see that $C_0$ contains both the oriented edge $(b)$ and the oriented path $[a^{-2}](a, b, a^{-1})$, so \cref{StandardAlteration} provides a hamiltonian cycle~$C_1$, such that 
	$$ \text{$(\voltage C_0)^{-1} (\voltage C_1)$ is a conjugate of $[a, b^{-1}][a, b^{-1}]^a$.} $$
Similarly, since $s_5= b$ and $s_1s_2s_3s_4 = c^2$, we see that $C_1$ contains both the oriented edge $[c^2](b)$ and the oriented path $[c^2a^{-2}](a, b, a^{-1})$, so \cref{StandardAlteration} provides a hamiltonian cycle~$C_2$, such that 
	$$ \text{$(\voltage C_1)^{-1} (\voltage C_2)$ is also a conjugate of $[a, b^{-1}][a, b^{-1}]^a$.} $$
Since no element of~$G$ inverts any nontrivial element of~$G'$ (recall that $|G|$ is odd), this implies that $(\voltage C_i)^{-1} (\voltage C_j)$ generates~$G'$ whenever $i \neq j$. So \fullcref{MarusicMethod34}{3} applies (since all three hamiltonian cycles contain the oriented edge $[s_1](s_2)$.
\end{proof}

\begin{RomanCase}
Assume $\#S = 2$.
\end{RomanCase}

\begin{proof}
We have $S = \{a,b\}$, so $|G| = 9p^\mu q^\nu$. We may assume $p,q > 3$, for otherwise \cref{G=27palpha} applies (perhaps after interchanging $p$ and~$q$). 

One very special case with a lengthy proof will be covered separately:

\begin{assump}
Assume \cref{9pq-hard} below does not provide a hamiltonian cycle in $\Cay(G;S)$.
\end{assump}

Under this assumption, we will always use \cref{FGL} with $N = G'$, so \cref{FreeLunch} allows us to assume $G' = \integer_{pq}$.

Let
	$$ C = (a^{-2}, b^{-1}, a, b^{-1},  a^{-2}, b^2) ,$$
so $C$ is a hamiltonian cycle in $\Cay(\quot{G}; a,b)$.
We have\refnotelower{S=2VoltageCNote}
	\begin{align} \label{S=2VoltageC}
	 \voltage C
	= a^{-2} b^{-1} a b^{-1} a^{-2} b^2 
	= [a, b]^a [a,b]  [a,b]^b  (a^{-3})^{b^2} 
	. \end{align}

Let $\underline{G} = G/\integer_p$, so $\underline{G}' = \integer_{q}$. 
Since $p,q > 3$, we know $\gcd \bigl( |\quot{G}|, |G'| \bigr) = 1$, so $G \iso \quot{G} \ltimes G'$ \cite[Thm.~6.2.1(i)]{Gorenstein-FinGrps}. Therefore $G' \cap Z(G)$ is trivial, so we may 
	$$ \text{assume that $a$ does not centralize $\integer_q$} $$
(perhaps after interchanging $a$ with~$b$). Therefore $a$ acts on~$\integer_q$ via a nontrivial cube root of unity. Since the nontrivial cube roots of unity are the roots of the polynomial $x^2 + x + 1$, this implies that $\underline{[a,b]^{a^{2}} [a, b]^a [a,b]} = \underline{e}$, so
	$$ \underline{[a, b]^a [a,b]}
	= (\underline{[a,b]^{a^{2}}})^{-1}
	= (\underline{[a,b]^{a^{-1}}})^{-1}
	 $$
(since $|\quot{a}| = 3$).
Furthermore, $\underline{a^{-3}} = \underline{e}$ (since $a$ has trivial centralizer in~$\integer_q$). Hence, 
	\begin{align*}
	 \underline{\voltage C}
	&= \underline{[a, b]}^a \,  \underline{[a,b]} \,  \underline{[a,b]}^b  (\underline{a}^{-3})^{b^2}
	\\&= (\underline{[a,b]}^{a^{-1}})^{-1} \, \underline{ [a,b]}^b  \, \underline{e} 
	\\&= (\underline{[a,b]}^{a^{-1}})^{-1} \, \underline{ [a,b]}^b 
	. \end{align*}
Therefore
	\begin{align} \label{S=2-bisainv}
	\text{$\underline{\voltage C} \neq \underline{e}$ \ unless $y^b = y^{a^{-1}}$ for all $y \in \integer_q$.}
	\end{align}
Hence, we may assume $\langle \voltage C \rangle$ contains~$\integer_q$ (by replacing $b$ with its inverse if necessary).

\begin{subcase} \label{aCentsZqSubcase}
Assume $a$ centralizes $\integer_p$. 
\end{subcase}
Since $G' \cap Z(G)$ is trivial, we know that $b$ does not centralize~$\integer_p$. 
Also, we may assume $\langle \voltage C \rangle \neq G'$, for otherwise \cref{FGL} applies.
Therefore $\voltage C$ must project trivially to~$\integer_p$. Fixing $r,k \in \integer$ with 
	$$ \text{$[a,b]^b = [a,b]^r$ \  and  \ $a^{-3} = [a,b]^k$} $$
(and using the fact that $r^2 + r + 1 \equiv 0 \pmod{p}$), we see from \pref{S=2VoltageC} that this means
	$$0
	 \equiv 1 + 1 + r + kr^2
	\equiv 1 - r^2 + k r^2
	\equiv r^2(r - 1 + k)
	 \pmod{p} ,$$
so 
	$$k \equiv 1 - r \pmod{p} .$$
Therefore $k \not\equiv 0 \pmod{p}$ (since $r$ is a primitive cube root of unity). 
Also, since $a$ centralizes~$\integer_p$, we have\refnotelower{aibi}
	$$ [a^{-1}, b^{-1}]^{-kr}
	\equiv  \bigl( [a, b^{-1}]^{-1} \bigr) ^{-kr}
	= \bigl( [a,b]^{b^{-1}} \bigr)^{-kr}
	= [a,b]^{-k} 
	= a^3
	= (a^{-1})^{-3}
	\pmod{\integer_q}
	.$$
Therefore, replacing $a$ and~$b$ with their inverses replaces $k$ with $-kr$ (modulo~$p$), and it obviously replaces $r$ with~$r^2$. Hence, we may assume that we also have
	$$ -kr \equiv 1 - r^2 \equiv r^3 - r^2 = -(1 - r)r^2 \equiv - kr^2 \pmod{p} , $$
so $r \equiv 1 \pmod{p}$. This contradicts the fact that $b$ does not centralize~$\integer_p$.

\begin{subcase}
Assume $a$ does not centralize $\integer_p$.
\end{subcase}
We may assume that the preceding \namecref{aCentsZqSubcase} does not apply when $a$ and~$b$ are interchanged (and perhaps $p$ and~$q$ are also interchanged). Therefore, we may assume that either
\noprelistbreak
	\begin{itemize}
	\item $b$ centralizes both $\integer_p$ and~$\integer_q$, in which case, interchanging $p$ and~$q$ in \pref{S=2-bisainv} tells us that $\voltage C$ projects nontrivially to both $\integer_p$ and~$\integer_q$, so \cref{FGL} applies,
	or
	\item $b$ has trivial centralizer in~$G'$.
	\end{itemize}
Henceforth, we assume $a$ and~$b$ both have trivial centralizer in~$G'$. 

We may assume $y^b = y^a$ for $y \in \integer_q$, by replacing $b$ with its inverse if necessary.
We may also assume $\langle \voltage C \rangle \neq G'$ (for otherwise \cref{FGL} applies). Since $\langle \voltage C \rangle$ contains~$\integer_q$, this means that $\langle \voltage C \rangle$ does not contain~$\integer_p$. By interchanging $p$ and~$q$ in \pref{S=2-bisainv}, we conclude that $x^b = x^{a^{-1}}$ for $x \in \integer_p$. We are now in the situation where  a hamiltonian cycle in $\Cay(G; a,b)$ is provided by \cref{9pq-hard} below.
\end{proof}

The remainder of this section proves \cref{9pq-hard}, by applying \cref{FGL} with $N = \integer_{q^\nu}$.  To this end, the following \namecref{TriangleHC} provides a hamiltonian cycle in $\Cay \bigl( G/\integer_{q^\nu}; S \bigr)$.

\begin{lem} \label{TriangleHC}
Assume
\noprelistbreak
	\begin{itemize}
	\item $G = \integer_{p^\mu} \rtimes (\integer_3 \times \integer_3)
	= \langle x \rangle \rtimes \bigl( \langle a \rangle \times \langle b_0 \rangle \bigr)$,
	with $p > 3$,
	\item $b = x b_0 $,
	\item $x^b = x^{a^{-1}} = x^r$, where $r$ is a primitive cube root of unity in~$\integer_{p^\mu}$,
	\item $k \in \integer$, such that
	\noprelistbreak
		\begin{itemize}
		\item $k \equiv 1 \pmod{3}$,
		\item $k \equiv r \pmod{p^\mu}$,
		and
		\item $0 \le k < 3p^\mu$,
		\end{itemize}
	\item $\ell$ is the multiplicative inverse of~$k$, modulo~$3p^\mu$ (and $0 \le \ell < 3 p^\mu$),
	\item $C =  \bigl( \,
	a, b^{-2}, 
	(a^{-1}, b^2)^{k-1},      a^{-2},      b^2,
	(a, b^{-2})^{\ell-k-1}, a^{-2}, (b^{-2}, a)^{3p^\mu-\ell - 1}
	\, \bigr)$,
	and
	\item $\widetilde C$ is the walk obtained from~$C$ by interchanging $a$ and~$b$, and also interchanging $k$ and~$\ell$.
	\end{itemize}
Then either $C$ or $\widetilde C$ is a hamiltonian cycle in $\Cay ( G; a, b )$.
\end{lem}

\begin{proof}
Define
	\begin{align*}
	v_{2i + \epsilon} &= (ba)^i b^\epsilon && \text{for $\epsilon \in \{0,1\}$} , \\
	w_j &= (ba)^j b^{-1}
	, \end{align*}
and let $V = \{ v_i \}$ and $W =  \{ w_j \}$.
Note that, since $x^{ab} = x$, we have $|ab| = 3 p^\mu$, so $\#V = 6 p^\mu$ and $\#W = 3 p^\mu$,
 so $G$ is the disjoint union of $V$ and~$W$.
With this in mind, it is easy to see that $C_1 = (b^{-2}, a)^{3p^\mu}$ is a hamiltonian cycle in $\Cay(G;a,b)$.

Removing the edges of the subpaths $(b^{-2})$ and $[(ba)^k](b^{-2}, a, b^{-2})$ from $C_1$ results in two paths:%
\noprelistbreak
	\begin{itemize}
	\item path $P_1$ from $b^{-2} = b$ to $(ba)^k$,
	and
	\item path $P_2$ from~$(ba)^{k+1}b$ to $e$
	(since $(ba)^k(b^{-2} a b^{-2}) = (ba)^k(b a b) = (ba)^{k+1}b$).
	\end{itemize}
The union of $P_1$ and~$P_2$ covers all the vertices of~$G$ \emph{except} the interior vertices of the removed subpaths, namely,
	$$ \text{all vertices except 
	$b^{-1}$, $(ba)^k b^{-1}$, $(ba)^k b$, $(ba)^{k+1}$, and
	$(ba)^{k+1} b^{-1}$} . $$
By ignoring~$y$ in calculation \pref{(ba)^k=aby} below, 
we see that 
	$b^{-1} a^{-1} = (a^{-1} b^{-1})^k $,
which means
	$$ ab = (ba)^k .$$
 Since $b^{-2} = b$, this implies
	$$a b^{-2} = (ba)^k .$$ 
Also, since $a^{-1} = a^2$, we have
	$$ b a^{-1} b^2
	= ba^2b^2
	= (ba)(ab) b
	= (ba) \bigl( (ba)^k \bigr) b
	= (ba)^{k+1} b
	 .$$
Therefore 
	$$ \text{$Q_1 = (a, b^{-2})$ is a path from the end of~$P_2$ to the end of~$P_1$} , $$
and
	$$ \text{$Q_2 = [b](a^{-1}, b^2)$ is a path from the start of~$P_1$ to the start of~$P_2$} . $$
So, letting $-P_1$ be the reverse of the walk~$P_1$, we see that
	$$ C_2 = Q_1 \cup -P_1 \cup Q_2 \cup P_2 $$
is a closed walk. 

Note that the interior vertices of~$Q_1$ are
	$$ a = (ab)b^{-1} = (ba)^k b^{-1} $$
and
	$$ a b^{-1} = (ab) b = (ba)^k b ,$$
and the interior vertices of~$Q_2$ are
	$$ b a^{-1} = ba^2 = (ba) (ab)b^{-1} = (ba) (ba)^k b^{-1} = (ba)^{k+1} b^{-1} $$
and
	$$ b a^{-1} b = \bigl( (ba)^{k+1} b^{-1} \bigr) b = (ba)^{k+1} .$$
These are all but one of the vertices that are not in the union of $P_1$ and~$P_2$, so
	$$ \text{$C_2$ is a cycle that covers every vertex except $b^{-1}$} .$$

Notice that the only $a$-edge removed from~$C_1$ is 
	$[(ba)^k b^{-2}](a) = [(ba)^k b](a)$.
Since 
	$$k^2 \equiv (r^2)^2 = r^4 \equiv r \not\equiv 1 \pmod{p^\mu} ,$$
and $\ell$ is the multiplicative inverse of~$k$, modulo~$3p^\mu$, 
we know $k \neq \ell$, so this removed edge is not equal to $[(ba)^\ell b](a)$.
Therefore $[(ba)^\ell b](a)$ is an edge of~$C_2$.  
Now, we create a walk~$C^*$ by removing this edge from~$C_2$,
and replacing it with the path $[(ba)^\ell b](a^{-2})$. Since
	$$ (ab)^\ell = \bigl( (ba)^k \bigr)^\ell = (ba)^{k\ell} = ba ,$$
we see that the interior vertex of this path is
	$$ [(ba)^\ell b] a^{-1} = [b (ab)^\ell] a^{-1} = [b (ba)] a^{-1} = b^2 = b^{-1} . $$
Therefore $C^*$ covers every vertex, so it is a hamiltonian cycle.

Since $ab = (ba)^k$ and $ba = (ab)^\ell$, it is obvious that interchanging $a$ and~$b$ will also interchange $k$ and~$\ell$. Therefore, we may assume $k < \ell$, by interchanging $a$ and~$b$ if necessary. Then the edge $[(ba)^\ell b](a)$ is in~$P_2$, rather than being in~$P_1$. If we let $P_2'$ be the path obtained by removing this edge from~$P_2$,
and replacing it with $[(ba)^\ell b](a^{-2})$, then we have
	\begin{align*}
	C
	&= \bigl( \,
	(a, b^{-2}), 
	\quad
	(a^{-1}, b^2)^{k-1}, a^{-1},
	\quad
	(a^{-1}, b^2),
	\quad
	(a, b^{-2})^{\ell-k-1}, a^{-2}, (b^{-2}, a)^{3p^\mu-\ell - 1}
	\, \bigr) 
	\\&=	 \ Q_1 \  \cup \ -P_1 \  \cup \ Q_2 \ \cup \ P_2'
	\\& = C^*
	\end{align*}
is a hamiltonian cycle in $\Cay(G; a,b)$.
\end{proof}

\begin{prop} \label{9pq-hard}
Assume
\noprelistbreak
	\begin{itemize}
	\item $\quot{G} \iso \integer_3 \times \integer_3$,
	\item $G' = \integer_{p^\mu} \times \integer_{q^\nu}$, with $p \neq q$ and $p,q > 3$,
	\item $S = \{a,b\}$ has only two elements,
	\item $a$ and~$b$ have trivial centralizer in~$G'$,
	and
	\item $ab$ centralizes $\integer_{p^\mu}$ and $a b^{-1}$ centralizes $\integer_{q^\nu}$.
	\end{itemize}
Then $\Cay ( G; a, b )$ has a hamiltonian cycle.
\end{prop}

\begin{proof}
Since $\gcd \bigl( |\quot{G}|, |G'| \bigr) = 1$, we have 
	$$ G \iso G' \rtimes \quot{G} \iso (\integer_{p^\mu} \times \integer_{q^\nu}) \rtimes (\integer_3 \times \integer_3) .$$
Write $\integer_{p^\mu} = \langle x \rangle$ and $\integer_{q^\nu} = \langle y \rangle$.
Since $a$ does not centralize any nontrivial element of~$G'$, we may assume $a \in \integer_3 \times \integer_3$ (after replacing it by a conjugate). Write $b = \gamma b_0$, with $\gamma \in G'$ and $b_0 \in \integer_3 \times \integer_3$. Since $\langle a,b \rangle = G$, we must have $\langle \gamma \rangle = G'$, so we may assume $\gamma = xy$; therefore $b = xyb_0$.

Choose $r \in \integer$ with $x^{a^{-1}} = x^r$. Since $|a| = 3$ and $a$ does not centralize any nontrivial element of~$\integer_{p^\mu}$, we know that $r$ is a primitive cube root of unity, modulo~$p^\mu$. Also, since $ab$ centralizes~$\integer_{p^\mu}$, we have $x^b = x^r$.

Define $k$ and~$\ell$ as in \cref{TriangleHC}. Then, letting $\underline{G} = G / \integer_{q^\nu}$ (and perhaps interchanging $a$ with~$b$), \cref{TriangleHC} tells us that
	$$ C = \bigl( a, b^{-2}, 
	(a^{-1}, b^2)^{k-1},      a^{-2},       b^2,
	(a, b^{-2})^{\ell-k-1}, a^{-2}, (b^{-2}, a)^{3p^\mu-\ell - 1} \bigr) $$
is a hamiltonian cycle in $\Cay \bigl( \underline{G} ; a, b \bigr)$.

To calculate the voltage of~$C$, choose $s \in \integer$ with 
	$y^a = y^s $,
and let
	$$ y_1 = y^{s^2-(1 + s + s^2 + \cdots + s^{k-1})}
	= y^{s^2 - 1} $$
(since $1 + s + s^2 \equiv 0 \pmod{q}$ and $k \equiv 1 \pmod{3}$),
and note that
	\begin{align}	\label{(ba)^k=aby}
	(a^{-1} b^{-1})^k
	&= \bigl( a^{-1} (xy b_0)^{-1} \bigr)^k
	\\&= \bigl( a^{-1} b_0^{-1} y^{-1} x^{-1} \bigr)^k
	\notag\\&= x^{-k} \bigl( a^{-1} b_0^{-1} y^{-1}  \bigr)^k
	&& \text{($x$ commutes with $a^{-1} b_0^{-1}$ and~$y$)}
	\notag\\&= x^{-r} \bigl( a^{-1} b_0^{-1})^k  y^{-(1 + s + s^2 + \cdots + s^{k-1})} 
	&& \begin{pmatrix}
	\text{$k \equiv r \pmod{p^\mu}$ and} \\
	\text{$y^{a^{-1} b_0^{-1}} = y^{a^2 b_0^2} = y^{s^4} = y^s$}
	\end{pmatrix}
	\notag\\&= x^{-r} b_0^{-1} a^{-1} y^{-s^2} y_1
	&& \begin{pmatrix}
	\text{$a$ and~$b_0$ commute, $k \equiv 1 \pmod{3}$,} \\
	\text{and definition of~$y_1$}
	\end{pmatrix}
	\notag\\&= b_0^{-1} x^{-1} y^{-1} a^{-1}  y_1
	&& \text{($x^r = x^{b_0}$ and $y^{s^2} = y^{a^2} = y^{a^{-1}}$)}
	\notag\\&= b^{-1} a^{-1}  y^{s^2-1}
		&& \text{($b = xy b_0$ and $y_1 = y^{s^2-1}$)}
	\notag.\end{align}
Therefore 
	\begin{align*}
	\voltage C
	&= a b^{-2}
	(a^{-1} b^2)^{k-1} a^{-2} b^2
	(a b^{-2})^{\ell-k-1} a^{-2} (b^{-2} a)^{3p^\mu-\ell - 1}
	\\&= a b (a^{-1} b^{-1})^{k-1} a b 
	\bigl( b (a b)^{\ell-k-1} a \bigr) (b a)^{3p^\mu-\ell - 1}
		&& \text{($|a| = |b| = 3$)}
	\\&= a b (a^{-1} b^{-1})^{k} (a^{-1} b^{-1})^{-1} a b  \bigl( ba \bigr)^{\ell-k} (b a)^{-\ell - 1}
		&& \text{($|ba| = 3 p^\mu$)}
	\\&= ab (a^{-1} b^{-1})^{k} (ba) a b  (ba)^{-k} (ba)^{-1}
	\\&= ab \bigl( b^{-1} a^{-1} y^{s^2-1} \bigr) b a^2 b  \bigl( b^{-1} a^{-1} y^{s^2-1} \bigr) (a^{-1} b^{-1})
		&& \begin{pmatrix}  (ba)^{-k} = (a^{-1} b^{-1})^k \\  \phantom{(ba)^{-k} } = b^{-1} a^{-1} y^{s^2-1} \end{pmatrix}
	\\&= y^{s^2-1} b a y^{s^2-1} a^{-1} b^{-1}
	\\&= y^{s^2-1} y^{(s^2-1)s}
		&& \text{($y^{a^{-1} b^{-1}} = y^{a^2 b^2} = y^{s^4} = y^s$)}
	\\&= y^{(s^2 - 1)(1+s)}
	. \end{align*}
Since $s$ is a primitive cube root of unity modulo~$q^{\nu}$, we know $s \not\equiv \pm1 \pmod{q}$. Therefore, the exponent of~$y$ is not divisible by~$q$, which means $\voltage C \notin \langle y^q \rangle$, so $\voltage C$ generates~$\integer_{q^\nu}$. Hence, \cref{FGL} provides the desired hamiltonian cycle in $\Cay(G; a,b)$.
\end{proof}


\newpage

\appendix

\makeatletter
\renewcommand{\@oddhead}{\em \hfil Appendix: Notes to aid the referee\hfil\hbox to 0pt{\hss\thepage}}
\renewcommand{\@evenhead}{\em \hbox to 0pt{\thepage\hss}\hfil Appendix: Notes to aid the referee\hfil}
\makeatother

\section{Notes to aid the referee}

\bigskip\hrule width\textwidth \bigbreak

\begin{aid} \label{27grp-regularPf}
We may assume $(G')^3$ is trivial (by modding it out), so $G' =  Z(G)$. Therefore $[a,b] \in G' =  Z(G)$, so we have $[b,a^2] = [b,a]^2$. We also have $[a,b]^3 = e$, since $(G')^3$ is trivial. Therefore
	\begin{align*}
	(ab)^3 
	&= (ab)(ab)(ab)
	=a^3 b^{a^2} b^a b
	= a^3 \bigl( b [b, a^2] \bigr) \bigl( b [b, a] \bigr) \bigl( b \bigr)
	\\&= a^3 \bigl( b [b, a]^2 \bigr) \bigl( b [b, a] \bigr) \bigl( b \bigr)
	= a^3 b^3 \,  [b, a]^3
	= a^3 b^3 \,  e
	= a^3 b^3
	. \end{align*}
\end{aid}

\begin{aid} \label{G'=p/ZNontrivial}
Since we are only trying to show that something is nontrivial, there is no harm in modding out~$Z$; thus, we may assume $Z$ is trivial. Note that:
	\begin{itemize}
	\item $Z \cap G'$ is trivial, since $Z$ is in the center, but $a$~does not centralize $G' = \integer_p$. So $G'$ is still nontrivial after we mod out~$Z$.
	\item Since $a^n \in G'Z = G'$, and $a$~obviously centralizes~$a^n$, we have $a^n = e$.
	\end{itemize}
Write $b = a^i \gamma$ with $\gamma \in G'Z = G'$. We have
	\begin{align*}
	\voltage C_1
	&= b a^{-(i-1)} b a^{n-i-1}
	\\&= (a^i \gamma) a^{-(i-1)} (a^i \gamma) a^{-i-1}
		&& \text{($b = a^i \gamma$ and $a^n = e$)}
	\\&= a^i \gamma a \gamma a^{-i-1}
	\\&=(\gamma^a \gamma)^{a^{-i-1}}
	\end{align*}
This is obviously nontrivial, since $a$ (being of odd order) cannot invert~$\gamma$.
From \cref{gVoltage}, we know $\voltage C_1 = (\voltage C_2)^a$, so
	$$ (\voltage C_1)^{-1} (\voltage C_2)
	= \bigl( (\voltage C_2)^a \bigr)^{-1} (\voltage C_2)
	= [a, \voltage C_2]
	\neq e ,$$
because $a$ does not centralize~$G'$.
\end{aid}

\begin{aid} \label{aG'=bG'Voltage}
Write $C_1 = [g](s_i)_{i=1}^n$, with $\alpha_1$ being the final edge, so $C_2 = [g] \bigl( (s_i)_{i=1}^{n-1}, b \bigr)$. Then \cref{gVoltage} tells us
	$$ (\voltage C_2)^g = \left( \prod_{i=1}^{n-1} s_i \right) b 
	= \left( \prod_{i=1}^{n-1} s_i \right) (a\gamma)
	= \left( \prod_{i=1}^{n} s_i \right) \, \gamma
	=  (\voltage C_1)^g \, \gamma
	. $$
A similar calculation applies to $(\voltage C_2)^{-1} (\voltage C_3)$.
Then
	$$(\voltage C_1)^{-1} (\voltage C_3) = \Bigl( (\voltage C_1)^{-1} (\voltage C_2) \Bigr) \Bigl( (\voltage C_2)^{-1} (\voltage C_3) \Bigr) = \gamma_1 \gamma_2 .$$
\end{aid}

\begin{aid} \label{exhaustive}
Let us briefly explain why these cases are exhaustive.

\setcounter{case}{0}

\begin{case}
Assume there exist $a,b \in S$, such that $\langle [a,b] \rangle = G'$.
\end{case}
We may assume $\quot{b} \notin \langle \quot{a} \rangle$ and $\quot{a} \notin \langle \quot{b} \rangle$, for otherwise \cref{bina} applies (perhaps after interchanging $a$ and~$b$).
Then we may assume $|\quot{a}| = |\quot{b}| = 3$, for otherwise \cref{bnotina} applies (perhaps after interchanging $a$ and~$b$). Furthermore, since $\quot{b} \notin \langle \quot{a} \rangle$, we obviously have $\langle \quot{a} \rangle \neq \langle \quot{b} \rangle$. So \cref{ab=G'} applies.

\begin{case}
Assume there exist $a,b,c \in S$, such that $\langle a, b, c \rangle' = G'$.
\end{case}
Since $\langle a, b, c \rangle' = G'$ is cyclic, we know 
	$$\bigl\langle [s,t] \mid s,t \in \{a,b,c\} \bigr\rangle = \langle a, b, c \rangle'  = \integer_{p^\mu} \times \integer_{q^\nu} $$
(see \cite[Lem.~3.12]{GhaderpourMorris-Nilpotent}). 
Therefore, for $r \in \{p,q\}$, there exist $x_p,y_p \in \{a,b,c\}$, such that $\integer_{r^*} \subseteq \langle [x_r,y_r] \rangle$. There cannot be four distinct elements of $\{a,b,c\}$, so we may assume $x_p = x_q$. Then, letting $a = x_p$, $b = y_p$, and $c = y_q$, we have $\integer_{p^\mu} \subseteq \langle [a,b] \rangle$ and $\integer_{q^\nu} \subseteq \langle [a,c] \rangle$.

We may assume $\quot{b} \notin \langle \quot{a} \rangle$, for otherwise \cref{bandcina} applies (perhaps after interchanging $a$ and~$b$). Now, either \cref{bcNotIna} or \cref{abcRemainder} applies, depending on whether $c \notin \langle \quot{a} \rangle$ or $c \in \langle \quot{a} \rangle$, respectively.

\begin{case}
Assume there do not exist $a,b,c \in S$, such that $\langle a, b, c \rangle' = G'$.
\end{case}
Then \cref{G'NotIn<abc>}  applies.
\end{aid}

\begin{aid} \label{b=agamma}
Since $\quot{b} = \quot{a}^k$ and $G' = \integer_p \times \integer_q$, we may write $b = a^k \gamma_1 \lambda_1$, for some $\gamma_1 \in \integer_p$ and $\lambda_1 \in \integer_q$.
We have $[a,b] \equiv e \pmod{\integer_p}$ (since $\langle [a,b] \rangle = \integer_p$), so
	$$ e \equiv [a, b] = [a, a^k \gamma_1 \lambda_1] = [a, \gamma_1 \lambda_1] \equiv  [a, \lambda_1] \quad \pmod{\integer_p} .$$
Since \cref{aCents->proper} tells us that $a$ does not centralize~$\integer_q$, this implies $\lambda_1 = e$. Therefore $b = a^k \gamma_1$, as claimed.  

Similarly, we have $c = a^\ell \gamma_2$, for some $\gamma_2 \in \integer_q$.
\end{aid}

\begin{aid} \label{Calculate:bandcina}
We  have
	\begin{align*}
	\voltage C_2
	&= (a^{-(\ell-1)} c) \,  (b a^{-(k-1)} b ) \, (a^{n-k-\ell-2} c)
	\\&= (a \gamma_2) \, (a^k \gamma_1 a \gamma_1) (a^{-k-2} \gamma_2)
	&& \begin{pmatrix} c = a^\ell \gamma_2, \\ b = a^k \gamma_1, \\ a^n = e \end{pmatrix}
	\\&=  \gamma_2^{a^{-1}} \, (a^{k+1} \gamma_1 \gamma_1^{a^{-1}}a^{-k-1}) \, \gamma_2
	\\&= (\gamma_1 \gamma_1^{a^{-1}})^{a^{-k-1}}   ( \gamma_2^{a^{-1}} \gamma_2)
	&& \text{($G'$ is abelian)}
	.\end{align*}
\end{aid}

\begin{aid} \label{GensG':bandcina}
Since $[a,b]$ is a generator of~$\integer_p$, it is nontrivial, so $b \neq a^k$. Therefore $\gamma_1$ is nontrivial, so it generates~$\integer_p$. Also, since $|G|$ is odd, we know $a$ does not invert~$\integer_p$. Therefore $\gamma_1 \gamma_1^{a^{-1}} \neq e$, so it also generates~$\integer_p$. Hence, the conjugate $(\gamma_1 \gamma_1^{a^{-1}})^{a^{-k-1}}$ is also a generator of~$\integer_p$.

Similarly, $( \gamma_2^{a^{-1}} \gamma_2)$ generates~$\integer_q$. So the product $(\gamma_1 \gamma_1^{a^{-1}})^{a^{-k-1}}( \gamma_2^{a^{-1}} \gamma_2)$ generates $\integer_p \times \integer_q = G$.
\end{aid}

\begin{aid} \label{voltXXp}
Write $X = [g](x_i)_{i=1}^n$, where $(x_i)_{i=1}^{2A-1} = \bigl( a^{A-1}, b, a^{-(A-1)} \bigr)$, and let 
	$\pi = \prod_{i=2A}^n x_i$, so
	$$\text{$(\voltage X)^g = \bigl( a^{A-1} b a^{-(A-1)} \bigr) \pi$ 
	\quad and \quad
	$(\voltage X^p)^g = \bigl( a^{-(A-1)} b a^{A-1} \bigr) \pi$}
	. $$
Then
	\begin{align*} 
	\Bigl( (\voltage X)^{-1} (\voltage X^p) \Bigr)^g
	&= \Bigl( \bigl( a^{A-1} b a^{-(A-1)} \bigr) \pi \Bigr)^{-1}
	\Bigl(  \bigl( a^{-(A-1)} b a^{A-1} \bigr) \pi \Bigr)
	\\&= \pi^{-1} \Bigl( \bigl( a^{A-1} b a^{-(A-1)} \bigr)^{-1}
	 \bigl( a^{-(A-1)} b a^{A-1} \bigr) \Bigr)  \pi
	 , \end{align*}
so $(\voltage X)^{-1} (\voltage X^p)$ is a conjugate of 
	$\bigl( a^{A-1} b a^{-(A-1)} \bigr)^{-1} \bigl( a^{-(A-1)} b a^{A-1} \bigr)$.

Also, we have
	\begin{align*}
	\bigl( a^{A-1} b a^{-(A-1)} \bigr)^{-1} & \bigl( a^{-(A-1)} b a^{A-1} \bigr)
	\\&= \bigl( a^{A-1} b^{-1} a^{-(A-1)} \bigr) \bigl( a^{-(A-1)} b a^{A-1} \bigr)
	\\&= a^{A-1} \bigl( b^{-1} a^{-(A-1)} b a^{A-1}\bigr) a^{-(A-1)} \bigl(b^{-1} a^{-(A-1)} b a^{A-1} \bigr)
	\\&= \bigl[ b , a^{A-1}\bigr]^{a^{1-A}} \bigl[ b, a^{A-1} \bigr]
	\\&= \bigl[ b , a^{A-1}\bigr]^{a} \bigl[ b, a^{A-1} \bigr]
	&& \hskip-1.25in \begin{pmatrix}  
	\text{$a^A \in G'$ and $G'$ is abelian,}
	\\ \text{so $a^A$ centralizes $G'$}
	\end{pmatrix}
	.\end{align*}
\end{aid}

\begin{aid} \label{3SubgrpLem}
We have
	\begin{align*}
	e 
	&= [a, b^{-1}, c]^b [b, c^{-1}, a]^c [c, a^{-1}, b]^a 
	&& \text{(Three-Subgroup Lemma)}
	\\&= [a, b^{-1}, c]^b \cdot e^c \cdot [c, a^{-1}, b]^a 
	\\&= [a, b^{-1}, c]^b [c, a^{-1}, b]^a 
	, \end{align*}
so
	$$ [a, b^{-1}, c]^b = \bigl( [c, a^{-1}, b]^a \bigr)^{-1} .$$
Since the left-hand side is in~$\integer_p$ and the right-hand side is in~$\integer_q$, we conclude that they are both in $\integer_p \cap \integer_q = \{e\}$. So $[a, b^{-1}, c] = [c, a^{-1}, b] = e$. Exactly the same argument applies with any or all of $a$, $b$, and~$c$ replaced by their inverses, so we have $[a, b, c] = [c, a, b] = e$. Since $\langle [a,b] \rangle = \integer_p$ and $\langle [c,a] \rangle = \langle [a,c] \rangle = \integer_q$, this implies that $c$ centralizes~$\integer_p$, and $b$~centralizes~$\integer_q$.
\end{aid}

\begin{aid} \label{SpSq}
Assume, for simplicity, that $(G')^{pq} = \{e\}$, so $G = \underline{G}$. Let
	$$ S_p = \{\, a \in S \mid \exists b \in S, \ \langle [a,b]\rangle = \integer_p \,\} \cup \bigl( S \cap Z(G) \bigr) $$
and
	$$ S_q = \{\, a \in S \mid \exists b \in S, \ \langle [a,b]\rangle = \integer_q \,\} . $$

For any $a \in S \smallsetminus Z(G)$, there is some $b \in S$, such that $[a,b] \neq e$. Since, by assumption, we have $\langle [a,b] \rangle \neq G'$, we must have either $\langle [a,b] \rangle = \integer_p$ or  $\langle [a,b] \rangle = \integer_q$. So $a \in S_p$ or $a \in S_q$. Therefore, $S_p \cup S_q = S$.

Suppose $a \in S_p \cap S_q$. Then there exist $b,c \in S$, such that $\langle [a,b]\rangle = \integer_p$ (because $a \in S_p$) and  $\langle [a,c]\rangle = \integer_q$ (because $a \in S_q$). Therefore $\langle [a,b], [a,c] \rangle = \integer_p \times \integer_q = G'$, which contradicts the assumption of this \namecref{G'NotIn<abc>}.

So $S_p$ and~$S_q$ do form a partition of~$S$. Furthermore, it is clear that both sets are nonempty, because $\langle\, [s,t] \mid s,t \in S \rangle = G' = \integer_p \times \integer_q$ \cite[Lem.~3.12]{GhaderpourMorris-Nilpotent}.
\end{aid}

\begin{aid} \label{Gaa}
\ \break
\centerline{\includegraphics{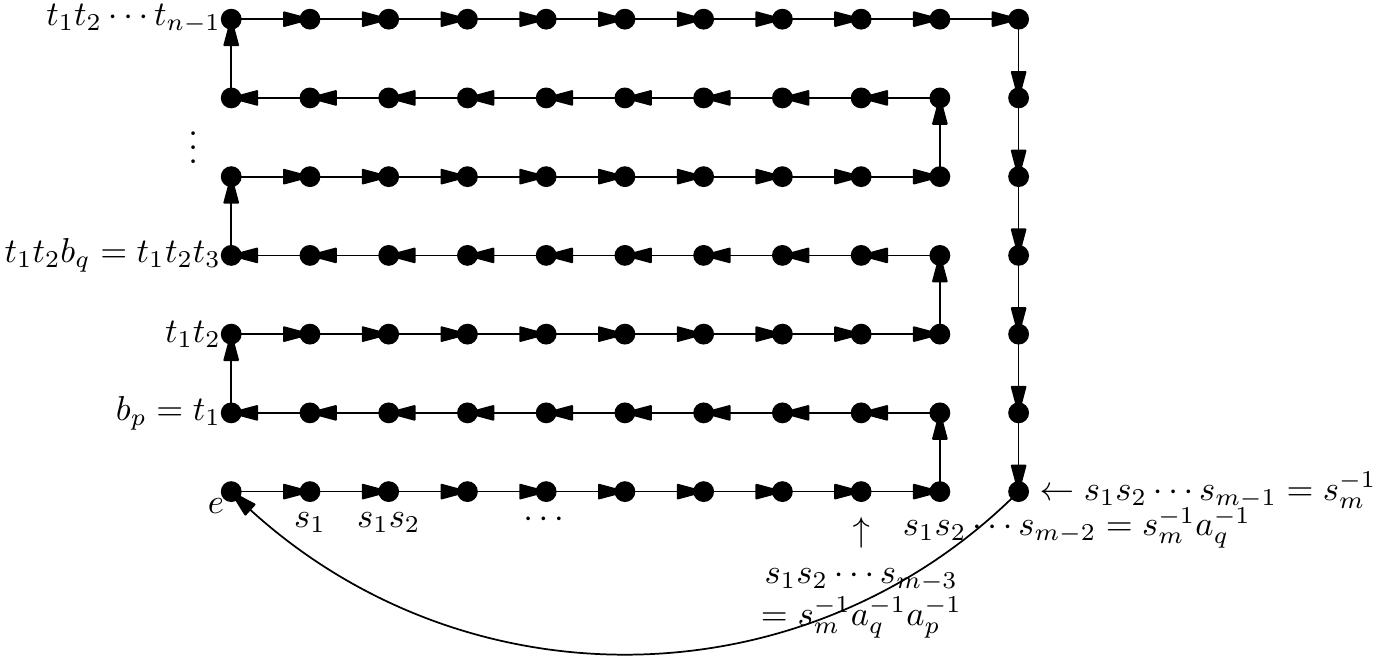}}
\end{aid}
%
%
%
%
%
%
%
%
%
%
%
%
%

\begin{aid} \label{GaaVoltage}
Let $C_0 = (u_i)_{i=1}^{mn}$, so
	$$ \voltage C_0 = u_1 u_2 \cdots u_{mn} .$$
To calculate $\voltage C_1$, we 
	\begin{itemize}
	\item replace some appearance of $b_p^{-1}$ with $a_q^{-1} b_p^{-1} a_q$,
	and
	\item replace some appearance of $a_p b_p a_p^{-1}$ with $b_p$.
	\end{itemize}
In other words, we multiply by the quantities 
	$$ \bigl( b_p^{-1} \bigr)^{-1} \bigl( a_q^{-1} b_p^{-1} a_q \bigr)
	= [b_p^{-1} , a_q ] $$
and
	$$ \bigl( a_p b_p a_p^{-1} \bigr)^{-1} \bigl( b_p \bigr) 
	= [a_p^{-1} , b_p] $$
at certain points, so
	\begin{align*}
	 \voltage C_1 
	 &= u_1 u_2 \cdots u_k  \,  [b_p^{-1} , a_q ] \, u_{k+1} \cdots u_\ell [a_p^{-1} , b_p] u_{\ell+1} \cdots u_{mn}
	\\&= u_1 u_2 \cdots u_{mn} [b_p^{-1} , a_q ]^g [a_p^{-1} , b_p]^h
	, \end{align*}
where $g = u_{k+1} u_{\ell+2} \cdots u_{mn}$ and $h = u_{\ell+1}  u_{\ell+2} \cdots u_{mn}$.
\end{aid}

\begin{aid} \label{S=2VoltageCNote}
	\begin{align*}
	 [a, b]^a [a,b]  [a,b]^b  (a^{-3})^{b^2} 
	 &= a^{-1} (a^{-1} b^{-1} a b) a \cdot (a^{-1} b^{-1} a b) \cdot b^{-1}(a^{-1} b^{-1} a b) b \cdot b^{-2} a^{-3} b^2
	 \\&= (a^{-1} a^{-1}) ( b^{-1} a) (b a  a^{-1} b^{-1}) (a b  b^{-1}a^{-1}) b^{-1} \bigl(  a (b b  b^{-2}) a^{-3} \bigr) b^2
	 \\&= (a^{-2}) (b^{-1} a) (e) (e) b^{-1} \bigl( a^{-2} \bigr) b^2
	 \\&= a^{-2} b^{-1} a b^{-1} a^{-2} b^2
	 \\&= \voltage C
	. \end{align*}
\end{aid}

\begin{aid} \label{aibi}
We have
	$$ [a^{-1}, b^{-1}] 
	= aba^{-1} b^{-1} 
	= a(ba^{-1} b^{-1} a) a^{-1}
	= [b^{-1},a]^{a^{-1}}
	= [b^{-1},a]
	= [a, b^{-1}]^{-1}$$
and
	$$ [a, b^{-1}]^{-1}
	= [b^{-1}, a]
	= b a^{-1} b^{-1} a
	= b (a^{-1} b^{-1} a b) b^{-1}
	= b [a, b]b^{-1}
	= [a, b]^{b^{-1}}
	.$$
\end{aid}

\end{document}